\documentclass[10pt, a4paper]{amsart}
\usepackage{amsfonts}
\usepackage{amssymb, amsmath}
\usepackage{bbm}
\usepackage{url}

\newtheorem{theorem}{Theorem}[section]
\newtheorem{proposition}[theorem]{Proposition}
\newtheorem{corollary}[theorem]{Corollary}
\newtheorem{lemma}[theorem]{Lemma}

\theoremstyle{definition}
\newtheorem{definition}[theorem]{Definition}
\newtheorem{example}[theorem]{\textbf{Example}}
\newtheorem{remark}[theorem]{\textbf{Remark}}
\newtheorem{remarks}[theorem]{\textbf{Remarks}}
\newtheorem{notation}[theorem]{\textbf{Notation}}

\title{An indispensable classification of monomial curves in $\mathbb{A}^4(\mathbbmss{k}) $}
\author{Anargyros Katsabekis \and Ignacio Ojeda}
\address{Centrum Wiskunde \& Informatica (CWI), Postbus 94079, 1090 GB Amsterdam, The Netherlands.} \email{katsabek@aegean.gr}
\address{Departamento de Matem\'aticas, Universidad de Extremadura,
E-06071 Badajoz (Spain).} \email{ojedamc@unex.es}
\keywords{Binomial ideal, toric ideal, monomial curve, minimal systems of generators, indispensable monomials, indispensable binomials.}
\thanks{The second author is partially supported by the project  MTM2012-36917-C03-01, National Plan I+D+I and by Junta de Extremadura (FEDER funds).}
\subjclass{13F20 (Primary) 16W50, 13F55 (Secondary).}

\begin{document}

\date{\today}

\begin{abstract}
In this paper a new classification of monomial curves in $\mathbb{A}^4(\mathbbmss{k})$ is given. Our classification relies on the detection of those binomials and monomials that have to appear in every system of binomial generators of the defining ideal of the monomial curve; these special binomials and monomials are called indispensable in the literature. This way to proceed has the advantage of producing a natural necessary and sufficient condition for the defining ideal of a monomial curve in $\mathbb{A}^4(\mathbbmss{k})$ to have a unique minimal system of binomial generators. Furthermore, some other interesting results on more general classes of binomial ideals with unique minimal system of binomial generators are obtained.
\end{abstract}

\maketitle

\section*{Introduction}

\par Let $\mathbbmss{k}[\mathbf{x}] := \mathbbmss{k}[x_1, \ldots, x_n]$ be the polynomial ring in $n$ variables over a field $\mathbbmss{k}.$ As usual, we will denote by $\mathbf{x}^\mathbf{u}$ the monomial $x_1^{u_1} \cdots x_n^{u_n}$ of $\mathbbmss{k}[\mathbf{x}] ,$ with $\mathbf{u} = (u_1, \ldots, u_n) \in \mathbb{N}^n,$ where $\mathbb{N}$ stands for the set of non-negative integers. Recall that a pure difference binomial ideal is an ideal of $\mathbbmss{k}[\mathbf{x}]$ generated by differences of monic monomials. Examples of pure difference binomial ideals are the toric ideals. Indeed, let $\mathcal{A} = \{\mathbf{a}_1, \ldots, \mathbf{a}_n\} \subset \mathbb{Z}^d$ and consider the semigroup homomorphism $\pi : \mathbbmss{k}[\mathbf{x}] \to \mathbbmss{k}[\mathcal{A}] := \bigoplus_{\mathbf{a} \in \mathcal{A}} \mathbbmss{k}\, \mathbf{t}^\mathbf{a};\ x_i \mapsto \mathbf{t}^{\mathbf{a}_i}.$ The kernel of $\pi$ is denoted by $I_\mathcal{A}$ and called the toric ideal of $\mathcal{A}.$ Notice that the toric ideal $I_\mathcal{A}$ is generated by all the binomials $\mathbf{x}^\mathbf{u} - \mathbf{x}^\mathbf{v}$ such that $\pi(\mathbf{x}^\mathbf{u}) = \pi(\mathbf{x}^\mathbf{v})$, see, for example, \cite[Lemma 4.1]{Sturmfels95}.

Defining ideals of monomial curves in the affine $n$-dimensional space $\mathbb{A}^n(\mathbbmss{k})$ serve as interesting examples of toric ideals. Of particular interest is to compute and describe a minimal generating set for such an ideal. In \cite{Herzog70} Herzog provides a minimal system of generators for the defining ideal of a monomial space curve. The case $n=4$ was treated by Bresisnky in \cite{Bresinsky88}, where Gr\"{o}bner bases techniques have been used to obtain a a minimal generating set of the ideal.

A recent topic arising in Algebraic Statistics is to study the problem when a toric ideal has a unique minimal system of binomial generators, see \cite{Charalambous07}, \cite{OjVi2}. To deal with this problem, Ohsugi and Hibi introduced in \cite{OhHi05} the notion of indispensable binomials, while Aoki, Takemura and Yoshida introduced in \cite{ATY08} the notion of indispensable monomials. The problem was considered for the case of defining ideals of monomial curves in \cite{GOj10}. Although this work offers useful information, the classification of the ideals having a unique minimal system of binomial generators remains an unsolved problem for $n \geq 4$. For monomial space curves Herzog's result provides an explicit classification of those defining ideals satisfying the above property. The aim of this work is to classify all defining ideals of monomial curves in $\mathbb{A}^4(\mathbbmss{k})$ having a unique minimal system of generators. Our approach is inspired by the classification made by Pilar Pis\'{o}n in her unpublished thesis.

The paper is organized as follows. In section 1 we study indispensable monomials and binomials of a pure difference binomial ideal. We provide a criterion for checking whether a monomial is indispensable, see Theorem \ref{Thm CombIndMonJ}, and also a sufficient condition for a binomial to be indispensable, see Theorem \ref{Thm CombIndBinJ}. As an application we prove that the binomial edge ideal of an undirected simple graph has a unique minimal system of binomial generators. Section 2 is devoted to special classes of binomial ideals contained in the defining ideal of a monomial curve. Corollary \ref{Cor Critical} underlines the significance of the critical ideal in the investigation of our problem. Theorem \ref{Indisp Circuits1} and Proposition \ref{Indisp Circuits2} provide  necessary and sufficient conditions for a circuit to be indispensable of the toric ideal, while Corollary \ref{Cor indispensablebinom3} will be particularly useful in the next section. In section 3 we study defining ideals of monomial curves in $\mathbb{A}^4(\mathbbmss{k})$. Theorem \ref{Thm Critical2} carries out a thorough analysis of a minimal generating set of the critical ideal. This analysis is used to derive a minimal generating set for the defining ideal of the monomial curve, see Theorem \ref{Th Main}. As a consequence we obtain the desired classification, see Theorem \ref{Thm Critical1}. Finally we prove that the defining ideal of a Gorenstein monomial curve in $\mathbb{A}^4(\mathbbmss{k})$ has a unique minimal system of binomial generators, under the hypothesis that the ideal is not a complete intersection.


\section{Generalities on indispensable monomials and binomials}

Let $\mathbbmss{k}[\mathbf{x}]$ be the polynomial ring over a field $\mathbbmss{k}.$ The following result is folklore, but for a lack of reference we sketch a proof.

\begin{theorem}\label{Th AgradJ}
Let $J \subset \mathbbmss{k}[\mathbf{x}]$ be a pure difference binomial ideal. There exist a positive integer $d$ and a vector configuration $\mathcal{A} = \{\mathbf{a}_1, \ldots, \mathbf{a}_n\} \subset \mathbb{Z}^d$ such that the toric ideal $I_\mathcal{A}$ is a minimal prime of $J.$
\end{theorem}

\begin{proof}
By \cite[Corollary 2.5]{Eisenbud96}, $\big( J : (x_1 \cdots x_n)^\infty \big)$ is a lattice ideal. More precisely, if $\mathcal{L} = \mathrm{span}_\mathbb{Z} \{ \mathbf{u} - \mathbf{v} \mid \mathbf{x}^\mathbf{u} - \mathbf{x}^\mathbf{v} \in J \},$ then $$\big( J : (x_1 \cdots x_n)^\infty \big) = \langle  \mathbf{x}^\mathbf{u} - \mathbf{x}^\mathbf{v} \mid \mathbf{u} - \mathbf{v} \in \mathcal{L} \rangle =: I_\mathcal{L} .$$ Now, by \cite[Corollary 2.2]{Eisenbud96}, the only minimal prime of $I_\mathcal{L}$ that is a pure difference binomial ideal is
$I_{\mathrm{Sat}(\mathcal{L}) } :=   \langle  \mathbf{x}^\mathbf{u} - \mathbf{x}^\mathbf{v} \mid \mathbf{u} - \mathbf{v} \in \mathrm{Sat}(\mathcal{L}) \rangle,$ where $\mathrm{Sat}(\mathcal{L}) := \{ \mathbf{u} \in \mathbb{Z}^n \mid z\, \mathbf{u} \in \mathcal{L}\ \text{for some}\ z \in \mathbb{Z}\}.$ Since $\mathbb{Z}^n/\mathrm{Sat}(\mathcal{L})  \cong \mathbb{Z}^d,$ for $d = n - \mathrm{rank}(\mathcal{L}),$ then $\mathbf{e}_i + \mathrm{Sat}(\mathcal{L}) = \mathbf{a}_i \in \mathbb{Z}^d,$ for every $i = 1, \ldots, n,$ and hence the toric ideal of $\mathcal{A} = \{ \mathbf{a}_1, \ldots, \mathbf{a}_n\}$ is equal to $I_{\mathrm{Sat}(\mathcal{L})}$ (see \cite[Lemma 12.2]{Sturmfels95}).

Finally, in order to see that $I_\mathcal{A}$ is a minimal prime of $J,$ it suffices to note that $J \subseteq P$ implies $\big( J : (x_1 \cdots x_n)^\infty \big) \subseteq P,$ for every prime ideal $P$ of $\mathbbmss{k}[\mathbf{x}].$
\end{proof}

\begin{remark}
Observe that if $J = \langle \mathbf{x}^{\mathbf{u}_j} - \mathbf{x}^{\mathbf{v}_j}\ \mid\ j = 1, \ldots, s \rangle,$ then $\mathcal{L} = \mathrm{span}_\mathbb{Z}\{ \mathbf{u}_j - \mathbf{v}_j\ \mid\ j = 1, \ldots, s \}.$ So, it is easy to see that, in general, $J \neq I_\mathcal{L}.$ For example, if $J = \langle x-y, z-t, y^2-yt \rangle,$ then $I_\mathcal{L} = \langle x-t, y-t, z-t \rangle.$
\end{remark}

Given a vector configuration $\mathcal{A} = \{ \mathbf{a}_1, \ldots, \mathbf{a}_n \} \subset \mathbb{Z}^d,$ we grade $\mathbbmss{k}[\mathbf{x}]$ by setting $\deg_\mathcal{A}(x_i) = \mathbf{a}_i,\ i = 1, \ldots, n.$  We define the $\mathcal{A}-$degree of a monomial $\mathbf{x}^\mathbf{u}$ to be $$\deg_\mathcal{A}(\mathbf{x}^\mathbf{u}) = u_1 \mathbf{a}_1 + \cdots + u_n \mathbf{a}_n.$$ A polynomial $f \in \mathbbmss{k}[\mathbf{x}]$ is $\mathcal{A}-$homogeneous if the $\mathcal{A}-$degrees of all the monomials that occur in $f$ are the same. An ideal $J \subset \mathbbmss{k}[\mathbf{x}]$ is $\mathcal{A}-$homogeneous if it is generated by $\mathcal{A}-$homogeneous polynomials. Notice that the toric ideal $I_\mathcal{A}$ is $\mathcal{A}-$homogeneous; indeed, by \cite[Lemma 4.1]{Sturmfels95}, a binomial $\mathbf{x}^\mathbf{u} - \mathbf{x}^\mathbf{v} \in I_\mathcal{A}$ if and only if it is $\mathcal{A}-$homogeneous.

The proof of the following result is straightforward.

\begin{corollary}\label{Cor AgradJ}
Let $J \subset \mathbbmss{k}[\mathbf{x}]$ be a pure difference binomial ideal and let $\mathcal{A} = \{\mathbf{a}_1, \ldots, \mathbf{a}_n\} \subset \mathbb{Z}^d.$ Then $J$ is $\mathcal{A}-$homogeneous if and only if $J \subseteq I_\mathcal{A}.$
\end{corollary}

Notice that the finest $\mathcal{A}-$grading on $\mathbbmss{k}[\mathbf{x}]$ such that a pure difference binomial ideal $J  \subset \mathbbmss{k}[\mathbf{x}]$ is $\mathcal{A}-$homogeneous occurs when $I_\mathcal{A}$ is a minimal prime of $J.$ Such an $\mathcal{A}-$grading does always exist by Theorem \ref{Th AgradJ}. Ideals with finest $\mathcal{A}-$grading are studied in much greater generality in \cite{KaAp10}.
An $\mathcal{A}-$grading on $\mathbbmss{k}[\mathbf{x}]$ such that a pure difference binomial ideal $J  \subset \mathbbmss{k}[\mathbf{x}]$ is $\mathcal{A}-$homogeneous is said to be positive if the quotient ring $\mathbbmss{k}[\mathbf{x}]/I_\mathcal{A}$ does not contain invertible elements or, equivalently, if the monoid $\mathbb{N} \mathcal{A}$ is free of units.

It is well known that \emph{the number of polynomials of $\mathcal{A}-$degree $\mathbf{b} \in \mathbb{N}\mathcal{A}$ in any minimal system of $\mathcal{A}$-homogeneous generators is $\dim_\mathbbmss{k} \mathrm{Tor}_1^R(\mathbbmss{k}, \mathbbmss{k}[\mathcal{A}])_\mathbf{b}$} (see, e.g. \cite[Chapter 12]{Sturmfels95}).
 Thus, we say that $I_\mathcal{A}$ has minimal generators in degree $\mathbf{b}$ when $\dim_\mathbbmss{k} \mathrm{Tor}_1^R(\mathbbmss{k}, \mathbbmss{k}[\mathcal{A}])_\mathbf{b} \neq 0.$ In this case, if $f \in I_\mathcal{A}$ has degree $\mathbf{b}$ we say that $f$ is a \textbf{minimal generator} of $I_\mathcal{A}.$
 
From now on, let $\mathcal{A} = \{\mathbf{a}_1, \ldots, \mathbf{a}_n\} \subset \mathbb{Z}^d$ be such that
the quotient ring $\mathbbmss{k}[\mathbf{x}]/I_\mathcal{A}$ does not contain invertible elements and let
$J  \subset \mathbbmss{k}[\mathbf{x}]$ be an $\mathcal{A}-$homogeneous pure difference binomial ideal.

\begin{definition}
A \textbf{binomial} $f = \mathbf{x}^\mathbf{u} - \mathbf{x}^\mathbf{v} \in J$ is called \textbf{indispensable} of $J$ if every system of binomial generators of $J$ contains $f$ or $-f,$ while a \textbf{monomial} $\mathbf{x}^\mathbf{u}$ is called \textbf{indispensable} of $J$ if every system of binomial generators of $J$ contains a binomial $f$ such that $\mathbf{x}^\mathbf{u}$ is a monomial of $f.$
\end{definition}

In the following we will write $M_J$ for the monomial ideal generated by all $\mathbf{x}^\mathbf{u}$ for which there exists a nonzero $\mathbf{x}^\mathbf{u} - \mathbf{x}^\mathbf{v} \in J.$

The next proposition is the natural generalization of \cite[Proposition 3.1]{Charalambous07}, but for completeness, we give a proof.

\begin{proposition}\label{Prop IndMonJ}
The indispensable monomials of $J$ are precisely the minimal generators of $M_J.$
\end{proposition}

\begin{proof}
Let $\{f_1, \ldots, f_s\}$ be a system of binomial generators of $J.$ Clearly, the monomials of the $f_i,\ i = 1, \ldots, s,$ generate $M_J.$ Let $\mathbf{x}^\mathbf{u}$ be a minimal generator of $M_J.$ Then $\mathbf{x}^\mathbf{u} - \mathbf{x}^\mathbf{v} \in J,$ for some nonzero $\mathbf{v} \in \mathbb{N}^n.$ Now, the minimality of $\mathbf{x}^\mathbf{u}$ assures that $\mathbf{x}^\mathbf{u}$ is a monomial of $f_j$ for some $j.$ Therefore every minimal generator of $M_J$ is an indispensable monomial of $J.$ Conversely, let $\mathbf{x}^\mathbf{u}$ be an indispensable monomial of $J.$ If $\mathbf{x}^\mathbf{u}$ is not a minimal generator of $M_J,$ then there is a minimal generator $\mathbf{x}^\mathbf{w}$ of $M_J$ such that $\mathbf{x}^\mathbf{u} = \mathbf{x}^{\mathbf{w}} \mathbf{x}^{\mathbf{u}'}$ with $\mathbf{u}' \neq \mathbf{0}.$ By the previous argument $\mathbf{x}^\mathbf{w}$ is an indispensable monomial of $J,$ hence without loss of generality we may suppose that $f_k = \mathbf{x}^\mathbf{w} - \mathbf{x}^\mathbf{z}$ for some $k$ and $\mathbf z \in \mathbb{N}^n$. Thus, if $f_j = \mathbf{x}^\mathbf{u} - \mathbf{x}^\mathbf{v},$ then $$f'_j = \mathbf{x}^{\mathbf{u}'} \mathbf{x}^\mathbf{z} - \mathbf{x}^\mathbf{v} =  f_j - \mathbf{x}^{\mathbf{u}'} f_k \in J$$ and therefore we can replace $f_j$ by $f'_j$ in $\{f_1, \ldots, f_s\}.$ By repeating this argument as many times as necessary, we will find a system of binomial generators of $J$ such that no element has $\mathbf{x}^\mathbf{u}$ as monomial, a contradiction to the fact that $\mathbf{x}^\mathbf{u}$ is indispensable.
\end{proof}

\begin{corollary}\label{Cor CombIndMinJ}
If $\mathbf{x}^\mathbf{u} \in M_J$ is an indispensable monomial of $I_\mathcal{A},$ then it is also an indispensable monomial of $J.$
\end{corollary}

\begin{proof}
It suffices to note that $M_J \subseteq M_{I_\mathcal{A}}$ by Corollary \ref{Cor AgradJ}.
\end{proof}

Now, we will give a combinatorial necessary and sufficient condition for a monomial $\mathbf{x}^\mathbf{u} \in \mathbbmss{k}[\mathbf{x}]$ to be indispensable of $J.$

\begin{definition}
For every $\mathbf{b} \in \mathbb{N}\mathcal{A}$ we define the graph $G_\mathbf{b}(J)$ whose vertices are the monomials of $M_J$ of $\mathcal{A}-$degree $\mathbf{b}$ and two vertices $\mathbf{x}^\mathbf{u}$ and $\mathbf{x}^\mathbf{v}$ are joined by an edge if
\begin{itemize}
\item[(a)] $\mathrm{gcd}(\mathbf{x}^\mathbf{u}, \mathbf{x}^\mathbf{v}) \neq 1;$
\item[(b)] there exists a monomial $1 \neq \mathbf{x}^\mathbf{w}$ dividing $\mathrm{gcd}(\mathbf{x}^\mathbf{u}, \mathbf{x}^\mathbf{v})$ such that the binomial $\mathbf{x}^{\mathbf{u}-\mathbf{w}} - \mathbf{x}^{\mathbf{v}-\mathbf{w}}$ belongs to $J.$
\end{itemize}
\end{definition}

Notice that $G_\mathbf{b}(J) = \varnothing$ exactly when $M_J$ has no element of  $\mathcal{A}-$degree $\mathbf{b};$ in particular, $G_\mathbf{b}(J) = \varnothing$ if  $\mathbf{b} = \mathbf{0},$ because $1 \not\in M_J$ (otherwise, $\mathbbmss{k}[\mathbf{x}]/I_\mathcal{A}$ would contain invertible elements). Moreover, since $J \subseteq I_\mathcal{A},$ we have that $G_\mathbf{b}(J)$ is a subgraph of $G_\mathbf{b}(I_\mathcal{A}),$ for all $\mathbf{b}.$ Finally, we observe that condition (b) is trivially fulfilled for $J = I_\mathcal{A}$ because $\big(  I_\mathcal{A} : (x_1 \cdots x_n)^\infty \big) = I_\mathcal{A},$ in this case, if $G_\mathbf{b}(J) \neq \varnothing,$ the graph $G_\mathbf{b}(J)$ is nothing but the $1-$skeleton of the simplicial complex $\nabla_\mathbf{b}$ appearing in \cite{OjVi2}. Thus, we have the following result.

\begin{theorem}\label{Th OjVi1}
Let $\mathbf{x}^\mathbf{u} - \mathbf{x}^\mathbf{v} \in I_\mathcal{A}$ be a binomial of $\mathcal{A}-$degree $\mathbf{b}.$ Then, $f$ is a minimal generator of $I_\mathcal{A}$ if and only if $\mathbf{x}^\mathbf{u}$ and $\mathbf{x}^\mathbf{v}$ lie in two different connected components of $G_\mathbf{b}(I_\mathcal{A}),$ in particular, the graph is disconnected.
\end{theorem}

\begin{proof}
For a proof see, for example, \cite[Section 2]{OjVi}.
\end{proof}

The next theorem provides a necessary and sufficient condition for a monomial to be indispensable of $J$.

\begin{theorem}\label{Thm CombIndMonJ}
A monomial $\mathbf{x}^\mathbf{u}$ is indispensable of $J$ if and only if $\{\mathbf{x}^\mathbf{u}\}$ is connected component of $G_\mathbf{b}(J)$, where $\mathbf{b} = \deg_\mathcal{A}(\mathbf{x}^\mathbf{u}).$
\end{theorem}

\begin{proof}
Suppose that $\mathbf{x}^\mathbf{u}$ is an indispensable monomial of $J$ and $\{\mathbf{x}^\mathbf{u}\}$ is not a connected component of $G_\mathbf{b}(J).$ Then, there exists $\mathbf{x}^\mathbf{v} \in M_J$ with $\mathcal{A}-$degree equal to $\mathbf{b}$ such that $\mathrm{gcd}(\mathbf{x}^\mathbf{u}, \mathbf{x}^\mathbf{v}) \neq 1$ and  $\mathbf{x}^{\mathbf{u}-\mathbf{w}} - \mathbf{x}^{\mathbf{v}-\mathbf{w}} \in J$, where $1 \neq \mathbf{x}^\mathbf{w}$ divides $\mathrm{gcd}(\mathbf{x}^\mathbf{u}, \mathbf{x}^\mathbf{v})$. So $\mathbf{x}^{\mathbf{u}-\mathbf{w}} \in M_J$ and properly divides $\mathbf{x}^\mathbf{u},$ a contradiction to the fact that $\mathbf{x}^\mathbf{u}$ is a minimal generator of $M_J$ (see Proposition  \ref{Prop IndMonJ}). Conversely, we assume that $\{\mathbf{x}^\mathbf{u}\}$ is connected component of $G_\mathbf{b}(J)$ with $\mathbf{b} = \deg_\mathcal{A}(\mathbf{x}^\mathbf{u})$ and that $\mathbf{x}^\mathbf{u}$ is not an indispensable monomial of $J.$ Then, by Proposition \ref{Prop IndMonJ}, there exists a binomial $f = \mathbf{x}^\mathbf{w} - \mathbf{x}^\mathbf{z} \in J,$ such that $\mathbf{x}^\mathbf{w}$ properly divides $\mathbf{x}^\mathbf{u}.$ Let $\mathbf{x}^\mathbf{u} = \mathbf{x}^{\mathbf{w}} \mathbf{x}^{\mathbf{u}'},$ then $1 \neq \mathbf{x}^{\mathbf{u}'}$ divides $\mathrm{gcd}( \mathbf{x}^\mathbf{u}, \mathbf{x}^{\mathbf{u}'} \mathbf{x}^{\mathbf{z}})$ and hence $(\mathbf{x}^\mathbf{u} - \mathbf{x}^{\mathbf{u}'} \mathbf{x}^{\mathbf{z}})/(\mathbf{x}^{\mathbf{u}'})  = f \in J.$ Thus, $\{ \mathbf{x}^\mathbf{u}, \mathbf{x}^{\mathbf{u}'} \mathbf{x}^{\mathbf{z}} \}$ is an edge of $G_\mathbf{b}(J),$ a contradiction to the fact that $\{ \mathbf{x}^\mathbf{u} \}$ is a connected component of $G_\mathbf{b}(J).$
\end{proof}

Now, we are able to give a sufficient condition for a binomial to be indispensable of $J$ by using our graphs $G_\mathbf{b}(J)$ (compare with \cite[Corollary 5]{GOj10}).

\begin{theorem}\label{Thm CombIndBinJ}
Given $\mathbf{x}^\mathbf{u} - \mathbf{x}^\mathbf{v} \in J$ and let $\mathbf{b}  = \deg_\mathcal{A}(\mathbf{x}^\mathbf{u}) \, \big(= \deg_\mathcal{A}(\mathbf{x}^\mathbf{v}) \big).$ If $G_\mathbf{b}(J) = \big\{ \{\mathbf{x}^\mathbf{u}\}, \{\mathbf{x}^\mathbf{v}\} \big\},$ then $\mathbf{x}^\mathbf{u} - \mathbf{x}^\mathbf{v}$ is an indispensable binomial of $J.$
\end{theorem}

\begin{proof}
Assume that $G_\mathbf{b}(J) = \big\{ \{\mathbf{x}^\mathbf{u}\}, \{\mathbf{x}^\mathbf{v}\} \big\}.$ Then, by Theorem \ref{Thm CombIndMonJ}, both $\mathbf{x}^\mathbf{u}$ and $\mathbf{x}^\mathbf{v}$ are indispensable monomials of $J.$ Let $\{f_1, \ldots, f_s\}$ be a system of binomial generators of $J.$ Since $\mathbf{x}^\mathbf{u}$ is an indispensable monomial, $f_i = \mathbf{x}^\mathbf{u} - \mathbf{x}^\mathbf{w} \neq 0,$ for some $i.$ Thus $\deg_\mathcal{A}(\mathbf{x}^\mathbf{u}) = \deg_\mathcal{A}(\mathbf{x}^\mathbf{w})$ and therefore $\mathbf{x}^\mathbf{w}$ is a vertex of $G_\mathbf{b}(J).$ Consequently, $\mathbf{w} = \mathbf{v}$ and we conclude that $\mathbf{x}^\mathbf{u} - \mathbf{x}^\mathbf{v}$ is an indispensable binomial of $J.$
\end{proof}

The converse of the above proposition is not true in general: consider for instance the ideal $J = \langle x-y, y^2 - yt,  z-t \rangle = \langle x - t, y - t, z - t \rangle \cap \langle x, y, z-t \rangle,$ then $J$ is $\mathcal{A}$-homogeneous for $\mathcal{A} = \{1, 1, 1, 1 \}.$ Both $x-y$ and $z-t$ are indispensable binomials of $J,$ while $G_\mathbf{1}(J) = \big\{ \{x\}, \{y\}, \{z\}, \{t\} \big\}.$

\begin{corollary}\label{Cor CombIndBinJ2}
If $f = \mathbf{x}^\mathbf{u} - \mathbf{x}^\mathbf{v} \in J$ is an indispensable binomial of $I_\mathcal{A},$ then $f$ is an indispensable binomial of $J.$
\end{corollary}

\begin{proof}
Let $\mathbf{b}= \deg_\mathcal{A}(\mathbf{x}^\mathbf{u}) \, \big(= \deg_\mathcal{A}(\mathbf{x}^\mathbf{v}) \big).$ By \cite[Corollary 7]{OjVi2}, if $\mathbf{x}^\mathbf{u} - \mathbf{x}^\mathbf{v}$ is an indispensable binomial of $I_\mathcal{A},$ then $G_\mathbf{b}(I_\mathcal{A}) = \big\{ \{ \mathbf{x}^\mathbf{u} \}, \{ \mathbf{x}^\mathbf{v} \} \big\}.$ Since $\mathbf{x}^\mathbf{u}$ and $\mathbf{x}^\mathbf{v}$ are vertices of $G_\mathbf{b}(J)$ and $G_\mathbf{b}(J)$ is a subgraph of $G_\mathbf{b}(I_\mathcal{A}) ,$ then  $G_\mathbf{b}(J) = G_\mathbf{b}(I_\mathcal{A})$ and therefore, by Theorem \ref{Thm CombIndBinJ}, we conclude that  $\mathbf{x}^\mathbf{u} - \mathbf{x}^\mathbf{v}$ is an indispensable binomial of $J.$
\end{proof}

Again we have that the converse is not true; for instance, $x-y$ and $z-t$ are indispensable binomials of $J = \langle x-y, y^2 - yt,  z-t \rangle$ and none of them is indispensable of the toric ideal $I_\mathcal{A}.$

We close this section by applying our results to show that the binomial edge ideals introduced in \cite{Herzog10} have unique minimal system of binomial generators.

Let $G$ be an undirected connected simple graph of the vertex set $\{1, \ldots, n\}$ and let $\mathbbmss{k}[\mathbf{x}, \mathbf{y}]$ be the polynomial ring in $2n$ variables, $x_1, \ldots, x_n,$ $y_1, \ldots, y_n,$ over $\mathbbmss{k}.$

\begin{definition}
The binomial edge ideal $J_G \subset  \mathbbmss{k}[\mathbf{x}, \mathbf{y}]$ associated to $G$ is the ideal generated by the binomials $f_{ij} = x_i y_j - x_j y_i,$ with $i < j,$ such that $\{i, j\}$ is an edge of $G.$
\end{definition}

Let $J_G \subset  \mathbbmss{k}[\mathbf{x}, \mathbf{y}]$ be the binomial edge ideal associated to $G.$ By definition, $J_G$ is contained in the determinantal ideal generated by the $2 \times 2-$minors of $$\left(\begin{array}{ccc} x_1 & \ldots & x_n \\ y_1 & \ldots & y_n \end{array}\right).$$ This ideal is nothing but the toric ideal associated to the Lawrence lifting, $\Lambda(\mathcal{A}),$ of $\mathcal{A} = \{1, \ldots, 1\}$ (see, e.g. \cite[Chapter 7]{Sturmfels95}).  Thus, $J_G \subseteq I_{\Lambda(\mathcal{A})} $ and the equality holds if and only if  $G$ is the complete graph on $n$ vertices. By the way, since $G$ is connected, the smallest toric ideal containing $J_G$ has codimension $n-1.$ So, the smallest toric ideal containing $J_G$ is $I_{\Lambda(\mathcal{A})},$ that is to say, $\Lambda(\mathcal{A})$ is the finest grading on $\mathbbmss{k}[\mathbf{x}, \mathbf{y}]$ such that $J_G$ is  $\Lambda(\mathcal{A})-$homogeneous.

\begin{corollary}
The binomial edge ideal $J_G$ has unique minimal system of binomial generators.
\end{corollary}

\begin{proof}
By \cite[Corollary 16]{OjVi2}, the toric ideal $I_{\Lambda(\mathcal{A})}$ is generated by its indispensable binomials, thus every $f_{ij} \in J_G,$  is an indispensable binomial of $I_{\Lambda(\mathcal{A})}.$ Now, by Corollary \ref{Cor CombIndBinJ2}, we conclude that $J_G$ is generated by its indispensable binomials.
\end{proof}

The above result can be viewed as a particular case of the following general result whose proof is also straightforward consequence of \cite[Corollary 16]{OjVi2} and Corollary \ref{Cor CombIndBinJ2}.

\begin{corollary}
Let $\mathcal{A} = \{\mathbf{a}_1, \ldots, \mathbf{a}_n \} \subseteq \mathbb{Z}^d$ be such that the monoid $\mathbb{N} \mathcal{A}$ is free of units. If $J \subseteq  \mathbbmss{k}[\mathbf{x}, \mathbf{y}]$ is a binomial ideal generated by a subset of the minimal system of binomial generators of $I_{\Lambda(\mathcal{A})} ,$ then $J$ has unique minimal system of binomial generators.
\end{corollary}

\section{Critical binomials, circuits and primitive binomials}

This section deals with binomial ideals contained in the defining ideal of a monomial curve. Special attention should be paid to the critical ideal; this is due to the fact that the ideal of a monomial space curve is equal to the critical ideal, see \cite{Herzog70} (see also the definition of neat numerical semigroup in \cite{komeda}). Throughout this section $\mathcal{A} = \{a_1, \ldots, a_n\}$ is a set of relatively prime positive integers and $I_\mathcal{A} \subset \mathbbmss{k}[\mathbf{x}] = \mathbbmss{k}[x_1, \ldots, x_n]$ is the defining ideal of the monomial curve $x_1 = t^{a_1}, \ldots, x_n = t^{a_n}$ in the $n-$dimensional affine space over $\mathbbmss{k}.$

\subsection{Critical binomials}

\begin{definition}
A binomial $x_i^{c_i} - \prod_{j \neq i} x_j^{u_{ij}} \in I_\mathcal{A}$ is called \textbf{critical} with respect to $x_i$ if $c_i$ is the least positive integer such that $c_i a_i \in \sum_{j \neq i} \mathbb{N} a_j.$
The \textbf{critical ideal} of $\mathcal{A}$, denoted by $C_\mathcal{A}$, is the ideal of $\mathbbmss{k}[\mathbf{x}]$ generated by all the critical binomials of $I_\mathcal{A}.$
\end{definition}

Observe that the critical ideal of $\mathcal{A}$ is $\mathcal{A}-$homogeneous.

\begin{notation}
From now on and for the rest of the paper, we will write $c_i$ for the least positive integer such that $c_i a_i \in \sum_{j \neq i} \mathbb{N} a_j,$ for each $i = 1, \ldots, n.$
\end{notation}

\begin{proposition}\label{Prop critical0}
The monomials $x_i^{c_i}$ are indispensable of $I_\mathcal{A},$ for every $i.$ Equivalently, $\{x_i^{c_i} \}$ is a connected component of $G_b(I_\mathcal{A}),$ where $b = c_i a_i,$ for every $i.$
\end{proposition}

\begin{proof}
The proof follows immediately from the minimality of $c_i$, Theorem \ref{Th OjVi1} and Theorem \ref{Thm CombIndMonJ}.
\end{proof}

The next proposition determines the indispensable critical binomials of the toric ideal $I_\mathcal{A}$.

\begin{theorem}\label{Thm Critical11}
Let $f = x_i^{c_i} - \prod_{j \neq i} x_j^{u_{ij}}$ be a critical binomial of $I_\mathcal{A},$ then $f$ is indispensable of $I_\mathcal{A}$ if, and only if, $f$ is indispensable of $C_\mathcal{A}.$
\end{theorem}

\begin{proof}
By Corollary  \ref{Cor CombIndBinJ2}, we have that if  $f$ is indispensable of $I_\mathcal{A}$, then it is indispensable of $C_\mathcal{A}.$ Conversely, assume that $f$ is indispensable of $C_\mathcal{A}.$ Let $\{f_1, \ldots, f_s\}$ be a system of binomial generators of $I_\mathcal{A}$ not containing $f.$ Then, by Proposition  \ref{Prop critical0},  $f_l = x_i^{c_i} - \prod_{j \neq i} x_j^{v_j}$ for some $l.$ So, $f_l$ is a critical binomial, that is to say, $f_l \in C_\mathcal{A}.$ Therefore, we may replace $f$ by $f_l$ and $f - f_l \in C_\mathcal{A}$ in a system of binomial generators of $C_\mathcal{A},$ a contradiction to the fact that $f$ is indispensable of $C_\mathcal{A}.$
\end{proof}

\begin{corollary}\label{Cor Critical}
If $I_\mathcal{A}$ has unique minimal system of binomial generators, then $C_\mathcal{A}$ also does.
\end{corollary}

\begin{proof}
The monomials $x_i^{c_i}$ are indispensable of $I_\mathcal{A},$ for each $i$ (see Proposition \ref{Prop critical0}). Thus, for every $i,$ there exists a unique binomial in $I_\mathcal{A}$ of the form $x_i^{c_i} - \prod_{j \neq i} x_j^{u_{ij}}$ and we conclude that $C_\mathcal{A}$ has unique minimal system of binomial generators.
\end{proof}

\begin{example}
Let $\mathcal{A}=\{4, 6, 2a+1, 2a+3\}$ where $a$ is a natural number. For $a=0,$ it is easy to see that $I_{\mathcal{A}}$ does not have a unique minimal system of binomial generators. If $a \geq 1,$ then $x_4^2 - x_1^a x_2$ and $x_4^2 - x_1 x_3^2 \in C_\mathcal{A}.$ Thus $C_\mathcal{A}$ is not generated by its indispensable binomials and therefore $I_{\mathcal{A}}$ does not have a unique minimal system of binomial generators.
\end{example}

\subsection{Circuits}\mbox{}\par

\medskip
Recall that the support of a monomial $\mathbf{x}^\mathbf{u}$ is defined to be the set $\mathrm{supp}(\mathbf{x}^\mathbf{u})=\{i \in \{1, \ldots, n\}\ \mid\ u_i \neq 0\}.$  The support of a binomial $f=\mathbf{x}^\mathbf{u} - \mathbf{x}^\mathbf{v} \in I_\mathcal{A}$, denoted by $\mathrm{supp}(f)$, is defined as the union $\mathrm{supp}(\mathbf{x}^\mathbf{u}) \cup \mathrm{supp}(\mathbf{x}^\mathbf{v}).$ We say that $f$ has full support when $\mathrm{supp}(f)=\{1,\ldots,n\}$.

\begin{definition}
An irreducible binomial $\mathbf{x}^\mathbf{u} - \mathbf{x}^ \mathbf{v} \in I_\mathcal{A}$ is called a \textbf{circuit} if its support is minimal with respect the inclusion.
\end{definition}

Recall that a polynomial in $\mathbbmss{k}[\mathbf{x}]$ is said to be \emph{irreducible} if it cannot be factored into the product of two (or more) non-trivial polynomials in $\mathbbmss{k}[\mathbf{x}]$.

\begin{lemma}\label{Lemma Circuits1}
Let $u_j(i) = \frac{a_i}{\mathrm{gcd}(a_i,a_j)},\ i \neq j.$ The set of circuits in $I_\mathcal{A}$ is equal to $$ \{x_i^{u_i(j)} - x_j^{u_j(i)} \mid i \neq j \}.$$
\end{lemma}

\begin{proof}
See \cite[Chapter 4]{Sturmfels95}
\end{proof}

The next theorem provides a class of toric ideals generated by critical binomials that, moreover, are circuits.

\begin{theorem}\label{Th Result}
If $C_\mathcal{A} = \langle x_1^{c_1} - x_2^{c_2}, \ldots, x_{n-1}^{c_{n-1}} -  x_{n}^{c_{n}} \rangle,$ then $C_\mathcal{A} = I_\mathcal{A}.$
\end{theorem}

\begin{proof} From the hypothesis the binomial $x_i^{c_i} - x_{i+1}^{c_{i+1}}$ belongs to $I_\mathcal{A}$, for each $i \in \{1,\ldots,n-1\}$. So, every circuit of $I_\mathcal{A}$ is of the form $x_k^{c_k}-x_l^{c_l}$, since $\mathrm{gcd}(c_k, c_l) = 1$. Now, from Proposition 2.2 in \cite{Villarreal}, the lattice $L=\ker_{\mathbb{Z}}(\mathcal{A})=\{{\bf u} \in \mathbb{Z}^n | u_1a_1+\ldots+u_na_n=0\}$ is generated by $\big\{ c_i \mathbf{e}_i - c_{j} \mathbf{e}_{j} \mid 1 \leq i \leq j \leq n \big\},$ where $\mathbf{e}_i$ is the vector with $1$ in the $i-$th position and zeros elsewhere. The rank of $L$ equals $n-1$ and a lattice basis is $\big\{{\bf v}_i=c_i \mathbf{e}_i - c_{i+1} \mathbf{e}_{i+1} \mid 1 \leq i \leq n-1 \big\}.$ Thus $C_\mathcal{A}$ is a lattice basis ideal. Let $M$ be the matrix with rows ${\bf v}_1,\ldots,{\bf v}_{n-1},$ then $M$ is a mixed dominating matrix and therefore, from Theorem 2.9 in \cite{Fischer}, the equality $C_\mathcal{A} = I_\mathcal{A}$ holds.
\end{proof}

\begin{remarks}\mbox{}\par
\begin{enumerate}
\item For $n = 4,$ a different proof of the above result can be found in \cite{Bresinsky75}.
\item The converse of Theorem \ref{Th Result} is not true in general (see, e.g., \cite{Villarreal}).
\item If every critical binomial of $I_\mathcal{A}$ is a circuit and the critical ideal has codimension $n-1,$ then $c_i a_i = c_ja_j ,$ for every $i \neq j.$ In particular, all minimal generators of $I_\mathcal{A}$ have the same $\mathcal{A}-$degree. This situation is explored in some detail in \cite{RGS11} from a semigroup viewpoint.
\end{enumerate}
\end{remarks}

The rest of this subsection is devoted to the investigation of necessary and sufficient conditions for a circuit to be indispensable of $I_\mathcal{A}$.

\begin{lemma}\label{Lemma Circuits2}
Let $f = x_i^{u_i(j)} - x_j^{u_j(i)} \in I_\mathcal{A}$ be a circuit and let $b = u_i(j) a_i.$ Then there is no monomial $\mathbf{x}^\mathbf{v}$ in the fiber $\deg_\mathcal{A}^{-1}(b)$ such that $\mathrm{supp}(\mathbf{x}^\mathbf{v})=\{i,j\}.$
\end{lemma}

\begin{proof}
Suppose to the contrary that there exists such a $\mathbf{v}$. Observe that $x_i^{u_i(j)} - x_j^{u_j(i)}$ is also a circuit of $I_{\{a_i/d, a_j/d\}},$ and $\mathbf{v} \in \deg^{-1}_{\{a_i/d,a_j/d\}}(b/d),$ with $d = \mathrm{gcd}(a_i, a_j)$. However, $\deg^{-1}_{\{a_i/d,a_j/d\}}(b/d) = \big\{x_i^{u_i(j)},  x_j^{u_j(i)} \big\}$ (see, for instance, \cite[Example 8.22]{GSRo}).
\end{proof}

\begin{theorem} \label{Indisp Circuits1}
Let $f = x_i^{u_i(j)} - x_j^{u_j(i)} \in I_\mathcal{A}$ be a circuit and let $b = u_i(j) a_i.$ Then, $f$ is indispensable of $I_\mathcal{A}$ if, and only if, $b-a_k \not\in \mathbb{N}\mathcal{A},$ for every $k \neq i, j.$ In particular, $u_i(j) = c_i$ and $u_j(i) = c_j.$
\end{theorem}

\begin{proof}
First of all, we observe that $\deg_\mathcal{A}^{-1}(b) \supseteq \big\{ x_i^{u_i(j)}, x_j^{u_j(i)} \big\}$ and equality holds if, and only if, $f$ is indispensable. So, the sufficiency condition follows. Conversely, since $b  \not\in \sum_{k \neq i, j} \mathbb{N} a_k,$ the supports of the monomials in $\deg_\mathcal{A}^{-1}(b)$ are included in $\{i,j\}$ and then, by Lemma \ref{Lemma Circuits2}, we are done. \end{proof}

Observe that from the above result it follows that \emph{if a circuit is indispensable, then it is a critical binomial}.

Let $\prec_{ij}$ be an $\mathcal{A}-$graded reverse lexicographical monomial order on $\mathbbmss{k}[\mathbf{x}]$ such that $x_k \prec_{ij} x_i$ and $x_k \prec_{ij} x_j$ for every $k \neq i,j.$

\begin{proposition} \label{Indisp Circuits2}
Let $f= x_i^{u_i(j)} - x_j^{u_j(i)} \in I_\mathcal{A}$ be a circuit. Then, $f$ is indispensable of $I_\mathcal{A}$ if, and only if, it belongs to the reduced Gr\"obner basis of $I_\mathcal{A}$ with respect to $\prec_{ij}.$
\end{proposition}

\begin{proof}
If $f$ is indispensable, then, from Theorem 13 in \cite{OjVi2}, it belongs to every Gr\"obner basis of $I_\mathcal{A}.$ Now, suppose that $f$ belongs to the reduced Gr\"obner basis of $I_\mathcal{A}$ with respect to $\prec_{ij}$ and it is not indispensable. Since $f$ is not indispensable, there exists a monomial $\mathbf{x}^{\mathbf{u}}$ in the fiber of $u_i(j) a_i$ different from $x_i^{u_i(j)}$ and $x_j^{u_j(i)}.$ By Lemma \ref{Lemma Circuits2}, we have that $\mathrm{supp}(\mathbf{x}^{\mathbf{u}}) \not\subset \{i,j\},$ so there is $k \in \mathrm{supp}(\mathbf{x}^{\mathbf{u}})$ and $k \not\in \{i,j\}$. Hence, both $f_i = x_i^{u_i(j)} - \mathbf{x}^{\mathbf{u}}$ and $f_j = x_j^{u_j(i)} - \mathbf{x}^{\mathbf{u}}$ belong to $I_\mathcal{A}$. Since the leading terms of $f_i$ and $f_j$ with respect to $\prec_{ij}$ equal to $x_i^{u_i(j)}$ and $x_j^{u_j(i)}$, respectively, we conclude
that $f = x_i^{u_i(j)} - x_j^{u_j(i)} \in I_\mathcal{A}$ is not in the reduced Gr\"obner basis of $I_\mathcal{A}$ with respect to $\prec_{ij},$ a contradiction.
\end{proof}

\subsection{Primitive binomials}

\begin{definition}
A binomial $\mathbf{x}^\mathbf{u} - \mathbf{x}^\mathbf{v} \in I_\mathcal{A}$ is called \textbf{primitive} if there exists no other binomial $\mathbf{x}^{\mathbf{u}'} - \mathbf{x}^{\mathbf{v}'}$ such that $\mathbf{x}^{\mathbf{u}'}$ divides $\mathbf{x}^\mathbf{u}$ and
$\mathbf{x}^{\mathbf{v}'}$ divides $\mathbf{x}^\mathbf{v}.$ The set of all primitive binomials is called the Graver basis of $\mathcal{A}$ and it is denoted by $Gr(\mathcal{A}).$
\end{definition}

\begin{theorem}\label{Th indispensablebinom3}
Let $f=x_i^{u_i}x_j^{u_j}-x_k^{u_k}x_l^{u_l} \in Gr(\mathcal{A})$ be such that $u_i < c_i, u_j < c_j, u_k < c_k$ and $u_l < c_l$ with $i,j,k$ and $l$ pairwise different. Then $f$ is indispensable of $J = I_\mathcal{A} \cap \mathbbmss{k}[x_i,x_j,x_k,x_l].$
\end{theorem}

\begin{proof}
By \cite[Proposition 4.13(a)]{Sturmfels95}, $J = I_\mathcal{A} \cap \mathbbmss{k}[x_i,x_j,x_k,x_l]$
is the toric ideal associated to $\mathcal{A}' = \{a_i, a_j, a_k, a_l\}.$ Thus, without loss of generality we may
assume $n=4,$ then $J = I_\mathcal{A}.$ We prove that $G_b(I_\mathcal{A}) = \{x_i^{u_i}x_j^{u_j}, x_k^{u_k}x_l^{u_l} \}$, where $b = u_i a_i + u_j a_j.$ Let
$\mathbf{x}^\mathbf{v} \in \deg_\mathcal{A}^{-1}(b)$ be different from $x_i^{u_i}x_j^{u_j}$ and
$x_k^{u_l}x_l^{u_l}.$ If $u_i < v_i,$ then $x_i^{u_i}(x_j^{u_j} - x_i^{v_i - u_i} x_j^{v_j} x_k^{v_k} x_l^{v_l})
\in I_\mathcal{A},$ thus $x_j^{u_j} - x_i^{v_i - u_i} x_j^{v_j} x_k^{v_k} x_l^{v_l} \in I_\mathcal{A}$ which is impossible by the minimality of $c_j$ (see Proposition \ref{Prop critical0}). Analogously, we can prove that $u_j \geq
v_j, u_k \geq v_k$ and $u_l \geq v_l.$ Therefore $x_i^{v_i} x_j^{v_j} (x_i^{u_i-v_i}x_j^{u_j-v_j} - x_k^{v_k}
x_l^{v_l}) \in I_\mathcal{A}$ and so $x_i^{u_i-v_i}x_j^{u_j-v_j} - x_k^{v_k} x_l^{v_l} \in I_\mathcal{A},$ a
contradiction with the fact that $f$ is primitive. This shows that $G_b(J) = \big\{ \{x_i^{u_i}x_j^{u_j}\}, \{ x_k^{u_k}x_l^{u_l} \} \big\}$ and, by Theorem \ref{Thm CombIndBinJ}, we
are done. \end{proof}

\begin{corollary}\label{Cor indispensablebinom3}
Let $f=x_i^{u_i}x_j^{u_j}-x_k^{u_k}x_l^{u_l} \in I_\mathcal{A}$ be such that $u_i < c_i$, $u_j < c_j$, $u_k>0$ and $u_l>0$ with $i,j,k$ and $l$ pairwise different. If $x_k^{u_k}x_l^{u_l}$ is indispensable of $J = I_\mathcal{A} \cap \mathbbmss{k}[x_i,x_j,x_k,x_l],$ then $f$ is indispensable of $J.$
\end{corollary}

\begin{proof}
Since, by Theorem \ref{Thm CombIndMonJ}, $\{x_k^{u_k}x_l^{u_l}\}$ is a connected component of $G_b(I_\mathcal{A}),$ where $b = u_k a_k + u_l a_l,$ the monomial $\mathbf{x}^\mathbf{v} \in \deg_\mathcal{A}^{-1}(b)$ in the above proof has its support in $\{i,j\}.$ Thus, repeating the arguments of the proof of Theorem \ref{Th indispensablebinom3}, we deduce that $u_i \geq
v_i$ and $u_j \geq
v_j$. But $x_i^{u_i}x_j^{u_j}-x_i^{v_i}x_j^{v_j} \in I_\mathcal{A}$, so $u_ia_i+u_ja_j=v_ia_i+v_ja_j$ which implies that $u_i=v_i$ and $u_j=v_j$. By Theorem \ref{Thm CombIndBinJ} we have that $f$ is indispensable of $J$.
\end{proof}

Combining Theorem \ref{Th indispensablebinom3} with Corollary \ref{Cor CombIndBinJ2} we get the following corollary.
\begin{corollary}
Given $i,j,k$ and $l \in \{1, \ldots, n\}$ pairwise different, let $J$ be the ideal of $\mathbbmss{k}[x_i,x_j,x_k,x_l]$ generated by all Graver binomials of $I_\mathcal{A}$ of the form $x_i^{u_i}x_j^{u_j}-x_k^{u_k}x_l^{u_l}$ with $u_i < c_i, u_j < c_j, u_k < c_k$ and $u_l < c_l.$ Then $J$ has unique minimal system of binomial generators.
\end{corollary}

Finally we provide another class of primitive binomials that are indispensable of a toric ideal.

\begin{corollary}
Let $f = x_i^{u_i}x_j^{u_j} - x_k^{u_k}x_l^{u_l} \in Gr(\mathcal{A})$ such that $0 < u_i < c_i$ and $0 < u_k < c_k,$ for $i, j, k$ and $l$ pairwise different. If $u_i a_i + u_j a_j$ is minimal among all Graver $\mathcal{A}-$degrees, then $f$ is indispensable of $I_\mathcal{A} \cap \mathbbmss{k}[x_i, x_j, x_l, x_k].$
\end{corollary}

\begin{proof}
Since $c_j a_j$ is a Graver $\mathcal{A}-$degree, we have $u_i a_i + u_j a_j \leq c_j a_j,$ so it follows $u_j < c_j.$ Similarly, we can prove $u_l < c_l.$ Therefore, by Theorem \ref{Th indispensablebinom3}, we conclude that $f$ is indispensable of $I_\mathcal{A} \cap \mathbbmss{k}[x_i, x_j, x_l, x_k].$
\end{proof}

It is worth to noting here that \cite[Theorem 6]{RGS11} offers a characterization of the family of affine semigroups for which $C_\mathcal{A} = Gr(\mathcal{A})$.

\section{Classification of monomial curves in $\mathbb{A}^4(\mathbbmss{k})$}

Let $\mathcal{A} = \{a_1, a_2, a_3, a_4\}$ be a set of relatively prime positive integers. First we will provide a minimal system of binomial generators for the critical ideal $C_\mathcal{A}$. This will be done by comparing the $\mathcal{A}$-degrees of the monomials $x_i^{c_i}$, for $i=1,\ldots,4$.

\begin{lemma} \label{Lemma Critical1}
Let $f_i = x_i^{c_i} - \prod_{j \neq i}x_j^{u_{ij}},\ i = 1, \ldots, 4,$ be a set of critical binomials of $I_\mathcal{A}$ and let $g_l \in I_\mathcal{A}$ be a critical binomial with respect  to $x_l$, for some $l \in \{1, \ldots, 4\}$. If $f_l \neq - f_i$ for every $i,$ then $g_l \in \langle f_1, f_2, f_3, f_4 \rangle.$
\end{lemma}

\begin{proof}
For simplicity we assume $l= 1.$ Let $g_1 = x_1^{c_1} - x_2^{v_2} x_3^{v_3} x_4^{v_4} \in I_\mathcal{A}$ be a  critical binomial. If $g_1 = f_1,$ there is nothing to prove. If $g_1 \neq f_1,$ without loss of generality we may  assume that $u_{12} > v_2, u_{13} \leq v_3$ and $u_{14} \leq v_4,$ so $g_1 - f_1=m_1g_2,$ with $m_1 =  x_2^{v_2} x_3^{u_{13}} x_4^{u_{14}}$ and $g_2 = x_2^{u_{12}-v_2} - x_3^{v_3-u_{13}} x_4^{v_4 - u_{14}} \in I_\mathcal{A}$  (in particular $u_{12}-v_2 \geq c_2$). But $x_1^{c_1} - x_1^{u_{21}}  x_2^{u_{12} -  c_2} x_3^{u_{13}+u_{23}} x_4^{u_{14}+u_{24}}
\in I_\mathcal{A}$ and also $f_1 \neq - f_2,$ thus from the minimality of $c_1$ it follows that $u_{21} = 0,$ that is to say, $f_2 \in \mathbbmss{k}[x_2, x_3, x_4]$.
For the sake of simplicity, write $g_2 = x_2^b - x_3^c x_4^d$ with $b,c,d \in \mathbb{N}$ and $b \geq c_2$. Hence $g_2 - x_2^{b-c_2} f_2 = x_2^{b-c_2} x_3^ {u_{23}} x_4^{u_{24}} - x_3^c x_4^d$. If $b - c_2 \geq c_2,$ we repeat the process. After a finite number of steps, $g_2 - h_{2} f_2 = x_2^ {b-kc_2} x_3^{k u_{23}} x_4^{k u_{24}} - x_3^c x_4^d$ with $0 \leq b - k c_2 < c_2$ and $h_2 \in \mathbbmss{k}[x_2, x_3, x_4]$. Then $(b-kc_2) a_2 + k u_{23} a_3 + k u_{24} a_4 = c a_3 + d a_4$. Since $0 \leq b - k c_2 < c_2$ then $x_3^{k u_{23}} x_4^{k u_{24}}$ does not divide $x_3^c x_4^d$. The case $x_3^c x_4^d$ divides $x_3^{k u_{23}} x_4^{k u_{24}}$ leads to $b = k c_2, c = k u_{23}$ and $d = k u_{24}$. In this setting, $g_2 = h_2 f_2, g_1 = f_1 + m_1h_2 f_2$ and we are done. The remaining cases are $k u_{23} \geq c$ and $d \geq k u_{24}$, or $k u_{23} \leq c$ and $d \geq k u_{24}$. Without loss of generality (by swapping variables if necessary), we may assume that $k u_{23} \leq c$ and $d \leq k u_{24}$. Hence $(b-k c_2) a_2 +  (k u_{24} -d ) a_4 = (c - k u_{23}) a_3$, and consequently $c - k u_{23}  \geq c_3$. We also deduce that $g_2 - h_2 f_2 = x_3^{k u_{23} }x_4^d (  x_2^{b - kc_2} x_4^{k u_{24} - d} - x_3^{c - k u_{23}}$. Set $m_3 = x_3^{k  u_{23}} x_4^d$ and $g_3 = x_2^{b - k c_2} x_4^{k u_{24} - d} - x_3^{c-k  u_{23}}$. Since $v_3 - u_{13} - k u_{23} = c - k u_{23} \geq c_3$, we have that $v_2 \geq c_3$. Thus
$x_1^{c_1} - x_1^{u_{31}} x_2^{u_{32} + v_2} x_3^{v_3-c_3} x_4^{u_{34} + v_4}
\in I_\mathcal{A}$ and $f_1 \neq -f_3,$ from the minimality of $c_1$ it follows that $u_{31} = 0,$ that is to say, $f_3 \in \mathbbmss{k}[x_2, x_3, x_4].$ Analogously, by using a similar argument as before (and by swapping variables $x_2$ and $x_4$, if necessary), we obtain $h_3 \in \mathbbmss{k}[x_2, x_3, x_4]$ such that either $g_3 = h_3 f_3$ or $g_3 - h_3 f_3 = m_3 g_4,$ with $m_3 = -x_2^{v'_2} x_4^{v''_4}, g_4 = x_4^{v'_4 - v_4 + u_{14}-v''_4} - x_2^{v''_2-v'_2} x_3^{v''_3}$ and $v''_3 < c_3.$ If $g_3 = h_3 f_3,$ then $g_1 = f_1 + m_1 h_2 f_2 + m_1 m_2 h_3 f_3$ and we are done. Otherwise, since $x_1^{c_1} - x_1^{u_{41}}  x_2^{v'_2+v_2+u_{42}} x_3^{v'_3+u_{13}+u_{43}} x_4^{v'_4+u_{14}-c_4}
\in I_\mathcal{A}$ and $f_1 \neq -f_4,$ from the minimality of $c_1$ it follows that $u_{41} = 0,$ that is to say, $f_4 \in \mathbbmss{k}[x_2, x_3, x_4].$ Therefore, we have that $f_2, f_3$ and $f_4 \in \mathbbmss{k}[x_2, x_3, x_4].$ Taking into account that $I_\mathcal{A} \cap \mathbbmss{k}[x_2, x_3, x_4]$ is generated by $f_2, f_3$ and $f_4$ (see, e.g., \cite[Proposition 4.13(a)]{Sturmfels95} and \cite[Theorem 2.2]{OjPis}), we conclude that $g_2 = g_{21} f_2 + g_{23} f_3 + g_{24} f_4$ and hence $g_1 = f_1 + m_1 g_{21} f_2 + m_1 g_{23} f_3 + m_1 g_{24} f_4,$ with $g_{2j} \in \mathbbmss{k}[x_2, x_3, x_4],\ j = 1,3,4.$
\end{proof}

\begin{proposition} \label{Prop Critical1b}
Let $f_i = x_i^{c_i} - \prod_{j \neq i} x_j^{u_{ij}},\ i = 1, \ldots, 4,$ be a set of critical binomials. If $f_i \neq - f_j$ for every $i \neq j,$ then $C_\mathcal{A} = \langle f_1, f_2, f_3, f_4 \rangle.$
\end{proposition}

\begin{proof}
The proof follows directly from Lemma \ref{Lemma Critical1}.
\end{proof}

Observe that $f_i = -f_j$ if and only if $f_i = x_i^{c_i} - x_j^{c_j}$ and $f_j=x_j^{c_j} - x_i^{c_i}$; in particular, they are circuits. The following proposition provides an upper bound for the minimal number of generators of the critical ideal.

\begin{proposition}
The minimal number of generators $\mu(C_\mathcal{A})$ of $C_\mathcal{A}$ is less than or equal to four.
\end{proposition}

\begin{proof}
Let $\mathcal{F}=\{f_1, \ldots, f_4\} \subset I_\mathcal{A}$ be such that $f_i$ is critical with respect to $x_i.$ If $f_i \neq -f_j,$ for every $i \neq j,$ then we are done by Proposition \ref{Prop Critical1b}. Otherwise, without loss of generality we may assume $f_1 = -f_2,$ that is to say, $f_1 = x_1^{c_1} - x_2^{c_2}.$ Suppose that $\mathcal{F}$ is not a generating set of $C_\mathcal{A}.$ We distinguish the following cases: (1) $f_1$ is indispensable of $I_\mathcal{A}$. Then there exists a critical binomial $g \in I_\mathcal{A}$ with respect to at least one of the variables $x_3$ and $x_4$, say $x_4$, such that $g \neq \pm f_i,$ for every $i.$ By substitution of $f_4$ with $g$ in $\mathcal{F}$ we have, from Lemma \ref{Lemma Critical1}, that every critical binomial with respect to $x_3$ or $x_4$ is in the ideal generated by the binomials of $\mathcal{F}$. Consequently the new set $\mathcal{F}$ generates $I_\mathcal{A}$.\\ (2) $f_1$ is not indispensable of $I_\mathcal{A}$. Then there exists a critical binomial $g \in I_\mathcal{A}$ with respect to al least one of the variables $x_1$ and $x_2$, for instance $x_2$, such that $g \neq \pm f_i,$ for every $i.$ We substitute $f_2$ with $g$ in $\mathcal{F}$. If $f_3 \neq -f_4,$ then we have, from Proposition \ref{Prop Critical1b}, that the new set $\mathcal{F}$ generates $I_\mathcal{A}$. Otherwise, we substitute $f_3$ with a critical binomial $h$ with respect to $x_3$ in $\mathcal{F}$ such that $h \neq \pm f_i,$ for every $i,$ when $f_3$ is not indispensable. So, in this case, $C_\mathcal{A}$ is generated by a set of $4$ critical binomials.
\end{proof}

\begin{lemma}\label{Lemma Critical2} If $c_ia_i \neq c_k a_k$ and $c_ia_i \neq c_l a_l,$ where $k \neq l,$ then either
the only critical binomial of $I_\mathcal{A}$ with respect to $x_i$ is $f =  x_i^{c_i} - x_j^{c_j}$ or there
exists a critical binomial $f \in I_\mathcal{A}$ with respect to $x_i$ such that $\mathrm{supp}(f)$ has
cardinality greater than or equal to three, where $\{i,j,k,l\}=\{1,2,3,4\}$.
\end{lemma}

\begin{proof}
Suppose the contrary and let $f_i= x_i^{c_i}-x_j^{u_j} \in I_\mathcal{A}$ where $u_j > c_j.$ We define  $f = x_i^{c_i}-x_i^{v_i}x_j^{u_j-c_j}x_k^{v_k}x_l^{v_l} = f_i + x_j^{u_j-c_j} f_j \in I_\mathcal{A}$ with $f_j = x_j^{c_j} - x_i^{v_i}x_k^{v_k}x_l^{v_l} \in I_\mathcal{A}.$ Now, from the minimality of $c_i$ it follows that $v_i = 0,$ thus at least one of $v_k$ or $v_l$ is different from zero since $f_j \in I_\mathcal{A}$, otherwise $f-f_i = x_j^{u_j} - x_j^{u_j-c_j} \in I_\mathcal{A}$, and this is impossible. Therefore we conclude that $\mathrm{supp}(f)$ has cardinality greater than or equal to $3,$ a contradiction. The cases $f_i = x_i^{c_i}-x_k^{u_k} \in I_\mathcal{A}$ and $f_i = x_i^{c_i}-x_l^{u_l} \in I_\mathcal{A}$ are analogous, by using that $c_i a_i \neq c_k a_k$ and $c_i a_i \neq c_l a_l$, respectively.
\end{proof}

\begin{lemma}\label{Lemma new1}
There is no minimal generating set of $C_\mathcal{A}$ of the form $\mathcal{S} = \{ x_i^{c_i} - x_j^{c_j}, x_j^{c_j} - \mathbf{x}^{\mathbf{u}_j}, x_k^{c_k} - x_l^{c_l}, x_l^{c_l} - \mathbf{x}^{\mathbf{u}_l} \}$, where $\{i,j,k,l\} = \{1,2,3,4\}$. In particular, if $c_i a_i= c_j a_j$ and $c_k a_k = c_l a_l,$ then $\mu(C_\mathcal{A}) < 4$.
\end{lemma}

\begin{proof}
Set $\mathbf{u}_j = (u_{j1}, \ldots, u_{j4})$ and $\mathbf{u}_l = (u_{l1}, \ldots, u_{l4})$. The minimality of $c_i,\ i \in \{1,2,3,4\}$, forces $u_{ji}=0=u_{jj}$, $0<u_{jk}<c_k$, $0<u_{jl}<c_l$, $0<u_{li}<c_i$, $0<u_{lj}<c_j$, $u_{lk}=0=u_{ll}$. 

Set $d_n=\gcd(\mathcal{A} \setminus \{a_n\}),\ n \in \{1,2,3,4\}$. By \cite[Theorem 3.10]{Herzog70}, the numerical semigroup generated by $\{a_i/d_l, a_j/d_l, a_k/d_l\}$ is symmetric and, from the proof of [Theorem 10.6,23], it is derived that $a_i/d_l=c_jc_k$, $a_j/d_l=c_ic_k$, $c_k=gcd(a_i/d_l,a_j/d_l)$ and $c_k a_k/d_l=u_{li}a_i/d_l+u_{lj}a_j/d_l$. Hence $a_i = c_j c_k d_l, a_j = c_i c_k d_l$ and $a_k = (u_{li} c_j + c_{lj} c_i) d_l$. Arguing analogously with $\{a_i/d_k, a_j/d_k, a_l/d_k\}$, we get $a_i = c_j c_l d_k, a_j = c_i c_l d_k$ and $a_l = (u_{li} c_j + c_{lj} c_i) d_l$. Thus, since $\mathrm{gcd}(c_i, c_j) = \gcd(c_k,c_l) = 1$, we conclude that $d_k = c_k$ and $d_l = c_l$. By considering now the symmetric semigroups $\{a_i/d_j, a_k/d_j, a_l/d_j\}$ and $\{a_j/d_i, a_k/d_i, a_l/d_i\}$, we get $a_i = (u_{jk} c_l + c_{jl} c_k) c_j, a_j = (u_{jk} c_l + c_{jl} c_k) c_i, a_k = c_i c_j c_l$ and $a_l = c_i c_j c_k$. 

Putting all this together, we obtain that $u_{jk} c_l + c_{jl} c_k = c_l c_k$ which forces either $u_{jk} = 0$ or $u_{jk} \geq c_k$, and this is a contradiction in both cases.
\end{proof}

\begin{theorem}\label{Thm Critical2} After permuting the variables, if necessary, there exists a minimal
system of binomial generators $\mathcal{S}$ of $C_\mathcal{A}$ of the following form: \begin{itemize}
\item[CASE 1:] If $c_i a_i \neq c_j a_j$, for every $\ i \neq j,$ then $\mathcal{S} = \{ x_i^{c_i} -
\mathbf{x}^{\mathbf{u}_i},\ i = 1, \ldots, 4 \}$ \item[CASE 2:] If $c_1 a_1= c_2 a_2$ and $c_3 a_3 = c_4 a_4,$
then either $c_2 a_2 \neq c_3 a_3$ and \begin{itemize}  \item[(a)] $\mathcal{S} = \{x_1^{c_1} - x_2^{c_2}, x_3^{c_3} - x_4^{c_4},
x_4^{c_4} - \mathbf{x}^{{u}_4} \}$ when $\mu(C_\mathcal{A})=3$ \item[(b)] $\mathcal{S} = \{x_1^{c_1} -
x_2^{c_2}, x_3^{c_3} - x_4^{c_4}\}$ when $\mu(C_\mathcal{A})=2$ \end{itemize} or $c_2 a_2 =c_3 a_3$ and
\begin{itemize} \item[(c)] $\mathcal{S} = \{x_1^{c_1} - x_2^{c_2}, x_2^{c_2} - x_3^{c_3}, x_3^{c_3} -
x_4^{c_4}\}$ \end{itemize} \smallskip \item[CASE 3:] If $c_1 a_1 = c_2 a_2 = c_3 a_3 \neq c_4 a_4,$ then
$\mathcal{S} = \{ x_1^{c_1} - x_2^{c_2}, x_2^{c_2} - x_3^{c_3}, x_4^{c_4} - \mathbf{x}^{\mathbf{u}_4}\}$
\item[CASE 4:] If $c_1 a_1 = c_2 a_2$ and $c_i a_i \neq c_j a_j$ for all $\{i, j\} \neq \{1, 2\},$ then
\begin{itemize} \item[(a)] $\mathcal{S} = \{x_1^{c_1} - x_2^{c_2}, x_i^{c_i} - \mathbf{x}^{{u}_i}\ \mid\ i =
2,3,4\}$ when $\mu(C_\mathcal{A})=4$ \item[(b)] $\mathcal{S} = \{x_1^{c_1} - x_2^{c_2}, x_i^{c_i} -
\mathbf{x}^{{u}_i}\ \mid\ i = 3,4\}$ when $\mu(C_\mathcal{A})=3$ \end{itemize} \end{itemize} where, in
each case, $\mathbf{x}^{\mathbf{u}_i}$ denotes an appropriate monomial whose support has cardinality greater than or equal to two.
\end{theorem}

\begin{proof} First, we observe that our assumption on the cardinality of $\mathbf{x}^{\mathbf{u}_i}$ follows
from Lemma \ref{Lemma Critical2}. We also notice that  $C_\mathcal{A}$ has no minimal generating set of the form $\mathcal{S} = \{ x_1^{c_1} - x_2^{c_2}, x_2^{c_2} - \mathbf{x}^{{u}_2}, x_3^{c_3} - x_4^{c_4}, x_4^{c_4} - \mathbf{x}^{{u}_4} \}$, by Lemma \ref{Lemma new1}.

Let $J$ be the ideal generated by $\mathcal{S}$. For the cases 1, 2(a-c), 3 and
4(a), it easily follows that $J = C_\mathcal{A}$  by Proposition \ref{Prop Critical1b}. Indeed, in
order to satisfy the hypothesis of Proposition \ref{Prop Critical1b}, we may take $f_4 =  x_4^{c_4} - x_1^{c_1}
\in J$ and $f_3 = x_3^{c_3} - x_1^{c_1} \in J$ in the cases 2(c) and 3, respectively. The cases 2(a) and 4(b)
happen when the only critical binomials of $I_\mathcal{A}$ with respect to $x_1$ and $x_2$ are $f_1 = x_1^{c_1} -
x_2^{c_2}$ and $f_2 = -f_1,$ respectively, then our claim follows from Lemma \ref{Lemma Critical1}. Furthermore, the
case 2(b) occurs when the only critical binomials of $I_\mathcal{A}$ are $\pm (x_1^{c_1} - x_2^{c_2})$ and $\pm
(x_3^{c_3} - x_4^{c_4}),$ so $J = C_\mathcal{A}$ by definition. On the other hand, since $x_i^{c_i}$ is an indispensable monomial of $I_\mathcal{A},$ for every $i,$ by Corollary
\ref{Cor CombIndMinJ}, we have that $x_i^{c_i}$ is an indispensable monomial of the ideal $J,$ for every $i.$
Then, we conclude that $\mathcal{S}$ is minimal in the sense that no proper subset of $\mathcal{S}$ generates
$J.$
\end{proof}

\begin{example}
This example illustrates all possible cases of Theorem \ref{Thm Critical2}.\par 
CASE 1: $\mathcal{A}=\{17,19,21,25\}$.\par 
CASE 2(a): $\mathcal{A}=\{30, 34, 42, 51\}$.\par 
CASE 2(b): $\mathcal{A}=\{39,91,100,350\}$.\par 
CASE 2(c): $\mathcal{A}=\{60,132,165,220\}$.\par 
CASE 3: $\mathcal{A}=\{12,19,20,30\}$.\par 
CASE 4(a): $\mathcal{A}=\{12,13,17,20\}$.\par 
CASE 4(b): $\mathcal{A}=\{4,6,11,13\}$.\\
The reader may perform the computations in detail by using the GAP package NumericalSgps (\cite{DGM}).
\end{example}

Since $C_\mathcal{A} \subseteq I_\mathcal{A},$ any minimal system of generators of $I_\mathcal{A}$ can not contain more than $4$ critical binomials. This provides an affirmative answer to the question after Corollary 2 in \cite{Bresinsky88}. Notice that the only cases in which $C_\mathcal{A}$ can have a unique minimal system of generators are 1, 2(b) and 4(b); in these cases $C_\mathcal{A}$ has a unique minimal system of binomial generators if
and only if the monomials $\mathbf{x}^{\mathbf{u}_i}$ are indispensable.

Now we focus our attention on finding a minimal set of binomial generators of $I_\mathcal{A}$, that will help us to solve the classification problem. The following lemma will be useful in the proof of Proposition \ref{Prop Critical3} and Theorem \ref{Th Main}.

\begin{lemma}\label{Lemma Critical3} (i) If $f = x_i^{u_i} - \mathbf{x}^\mathbf{v}$ is a minimal generator of $I_\mathcal{A}$ that is not critical, then there exists $j \neq i$ such that $\mathrm{supp}(\mathbf{x}^\mathbf{v}) \cap \{i,j\} = \varnothing$ and $c_i a_i = c_j a_j.$ Moreover, if $\mathbf{x}^\mathbf{v}$ is not indispensable, then $c_k a_k = c_l a_l,$ with $\{i,j,k,l\} = \{1,2,3,4\}.$\\ (ii) If $f = x_i^{u_i} x_j^{u_j} - \mathbf{x}^\mathbf{v}$ is a minimal generator of $I_\mathcal{A}$ with $u_i \neq 0$ and $u_j \geq c_j,$ then $\mathrm{supp}(\mathbf{x}^\mathbf{v})  \cap \{i,j\} = \varnothing$ and $c_i a_i = c_j a_j.$ In addition, if $\mathbf{x}^\mathbf{v}$ is not indispensable, then $c_k a_k = c_l a_l,$ with $\{i,j,k,l\} = \{1,2,3,4\}.$
\end{lemma}

\begin{proof} (i) Let $b = c_i a_i.$ Since $f$ is not a critical binomial, we have that $u_i > c_i.$ If $c_i a_i \neq c_j a_j,$ for every $j \neq i,$ then, from Lemma \ref{Lemma Critical2}, there exists a critical binomial $f = x_i^{c_i} - \mathbf{x}^\mathbf{w} \in I_\mathcal{A}$ such that $\mathrm{supp}(\mathbf{x}^\mathbf{w})$ has cardinality greater than or equal to two. If $\mathrm{supp}(\mathbf{x}^\mathbf{v}) \cap \mathrm{supp}(\mathbf{x}^\mathbf{w}) \neq
\varnothing,$ then $x_i^{u_i} \longleftrightarrow x_i^{u_i-c_i} \mathbf{x}^\mathbf{w} \longleftrightarrow
\mathbf{x}^\mathbf{v}$ is a path in $G_b(I_\mathcal{A}),$ a contradiction to the fact that $f$ is a minimal
generator by Theorem \ref{Th OjVi1}. Hence $\mathrm{supp}(\mathbf{x}^\mathbf{v}) \cap
\mathrm{supp}(\mathbf{x}^\mathbf{w})=\varnothing.$ We have that $\mathrm{supp}(\mathbf{x}^{\mathbf{v}+\mathbf{w}}) \subseteq \{j,k,l\},\ \mathrm{supp}(\mathbf{x}^\mathbf{v}) \cap
\mathrm{supp}(\mathbf{x}^\mathbf{w})=\varnothing$ and the cardinality of $
\mathrm{supp}(\mathbf{x}^\mathbf{w})$ is at least two. This implies that $\mathbf{x}^\mathbf{v}$ is a power of a variable, say $\mathbf{x}^\mathbf{v}=x_l^{v_l}$.
Observe that $v_l \geq c_l$ and as $f$ is not a critical binomial, $v_l \neq c_l$, whence $\mathbf{x}^\mathbf{z} = x_l^{v_l-c_l} x_i^{u_{li}} x_k^{u_{lk}} \in \deg^{-1}_\mathcal{A}(b)$ is a monomial such that $\mathrm{supp}(\mathbf{x}^\mathbf{z})$ has cardinality greater than or equal to $2$ and $l \in \mathrm{supp}(\mathbf{x}^\mathbf{z})$.
Then $x_i^{u_i}
\longleftrightarrow x_i^{u_i-c_i} \mathbf{x}^\mathbf{w} \longleftrightarrow \mathbf{x}^\mathbf{z}
\longleftrightarrow \mathbf{x}^\mathbf{v}$ is a path in $G_b(I_\mathcal{A}),$ a contradiction. Thus $c_i a_i = c_j a_j,$ for an $j \neq i.$ We have that $\mathrm{supp}(\mathbf{x}^\mathbf{v}) \cap \{i,j\} = \varnothing;$ otherwise $x_i^{u_i} \longleftrightarrow  x_i^{u_i-c_i} x_j^{c_j} \longleftrightarrow  \mathbf{x}^\mathbf{v}$ is a path in $G_b(I_\mathcal{A}),$ a contradiction again.

Finally, if $\mathbf{x}^\mathbf{v}$ is not indispensable, then, by Theorem \ref{Thm CombIndMonJ}, there exists a monomial $\mathbf{x}^\mathbf{w} \in \mathrm{deg}_\mathcal{A}^{-1}(b) \setminus \{\mathbf{x}^\mathbf{v} \}$ such that $\mathrm{supp}(\mathbf{x}^\mathbf{w}) \cap \mathrm{supp}(\mathbf{x}^\mathbf{v}) \neq \varnothing.$ If $j \in \mathrm{supp}(\mathbf{x}^\mathbf{w}),$ then $x_i^{u_i} \longleftrightarrow x_i^{u_i-c_i}x_j^{c_j} \longleftrightarrow  \mathbf{x}^\mathbf{w} \longleftrightarrow \mathbf{x}^\mathbf{v}$ is a path in $G_b(I_\mathcal{A}),$ a contradiction to the fact that $f$ is a minimal generator. Moreover $i \notin \mathrm{supp}(\mathbf{x}^\mathbf{w}),$ 
by the minimality of $c_i$. Thus $\mathrm{supp}(\mathbf{x}^\mathbf{w}) \subseteq \{k,l\}$ and also $x_k^{v_k}x_l^{v_l}-x_k^{w_k}x_l^{w_l} \in I_\mathcal{A}$. Suppose that $c_ka_k \neq c_la_l$
Then $v_k a_k + v_l a_l = w_k a_k + w_l a_l$. Assume without loss of generality that $w_l \geq v_l.$ We have that $(v_k - w_k) a_k = (w_l - v_l) a_l \neq 0$. Hence $v_k - w_k \geq c_k$. If $w_k \neq 0$, then $v_k > c_k$. If $w_k = 0, v_k a_k = (w_l - v_l) a_l$ and $v_l \neq 0$, since $\mathrm{supp}(\mathbf{x}^\mathbf{w}) \cap \mathrm{supp}(\mathbf{x}^\mathbf{v}) \neq \varnothing$. Thus $w_l - v_l \geq c_l$ and $w_l > c_l$. By using similar arguments as in the first part of the proof we arrive at a contradiction. Consequently $c_ka_k=c_la_l$. \\(ii) The proof is an easy adaptation of the arguments used in (i).
\end{proof}

\emph{For the rest of this section we keep the same notation as in Theorem \ref{Thm Critical2}}.

The following result was first proved by Bresinsky (see \cite[Theorem 3]{Bresinsky88}), but our argument seems to be shorter and more appropriate in our context.

\begin{proposition}\label{Prop Critical3}
There exists a minimal system of binomial generators of $I_\mathcal{A}$ consisting of the union of $\mathcal{S}$ and a set of binomials in $I_\mathcal{A}$ with full support.
\end{proposition}

\begin{proof}
By Lemma \ref{Lemma Critical3} (i), if for instance $f = x_i^{u_i} - \mathbf{x}^\mathbf{v}$ is 
in a minimal generating set of $I_\mathcal{A}$ and it is not a critical binomial with respect to any variable, then $c_ia_i=c_ja_j$, for $j \neq i$. We replace $f$ by $g = f - x_i^{u_i - c_i} (x_i^{c_i} - x_j^{c_j}) = x_i^{u_i-c_i} x_j^{c_j} - \mathbf{x}^\mathbf{v} \in I_\mathcal{A}$ in the minimal generating set of of $I_\mathcal{A}.$ Moreover, either $\mathrm{supp}(\mathbf{x}^\mathbf{v}) = \{k,l\}$ and $\{k,l\} \cap \{i,j\} = \varnothing,$ so $g$ has full support, or $\mathbf{x}^\mathbf{v}$ is a power of a variable, say $\mathbf{x}^\mathbf{v}=x_k^{v_k}$, with $v_k > c_k.$ In this case, by using again Lemma \ref{Lemma Critical3} (i), we replace $g$ with  $h = g + x_k^{v_k - c_k} (x_k^{c_k} - x_l^{c_l}) =
x_i^{u_i-c_i} x_j^{c_j}-x_k^{v_k - c_k} x_l^{c_l} \in I_\mathcal{A}$ with $\{k,l\} \cap \{i,j\} = \varnothing.$ Hence, there exists a system of generators of $I_\mathcal{A}$ consisting of the union of a system of binomials generators of $C_\mathcal{A}$ and a set  $\mathcal{S}'$ of binomials in $I_\mathcal{A}$ with full support. Furthermore, by Theorem \ref{Thm Critical2}, we may assume that $\mathcal{S}$ is a system of binomials generators of $C_\mathcal{A}$.

Now, let $f = x_i^{c_i} - \mathbf{x}^\mathbf{u} \in \mathcal{S}$ and suppose that $f = \sum_{n=1}^s g_n f_n$ where every $f_n \in (\mathcal{S} \setminus \{f\} ) \cup \mathcal{S}'.$ From the minimality of $c_i$ we have that $f_n = \pm (x_i^{c_i} - \mathbf{x}^\mathbf{v})$ and $|g_n|= 1,$ for some $n.$ Then, according to the cases in Theorem \ref{Thm Critical2}, either $\mathbf{x}^\mathbf{u}$ or $\mathbf{x}^\mathbf{v}$ is equal to $x_j^{c_j},$ for some $j \neq i.$ Now in the above expression of $f$ the term $x_j^{c_j}$ should be canceled, so, from the minimality of $c_j,$ we have $f_m = \pm (x_j^{c_j} - \mathbf{x}^\mathbf{w})$ and $|g_m|= 1,$ for an $m \neq n.$ Therefore, we conclude that either $\{x_i^{c_i} - x_j^{c_j}, \pm (x_i^{c_i} - \mathbf{x}^\mathbf{v}), \pm(x_j^{c_j} - \mathbf{x}^\mathbf{w})\}$ or $\{x_i^{c_i} - \mathbf{x}^\mathbf{u}, \pm(x_i^{c_i} - x_j^{c_j}), \pm(x_j^{c_j} - \mathbf{x}^\mathbf{w})\}$ is a subset of $\mathcal{S}.$ So, the only possible case is $\mathcal{S} = \{x_1^{c_1} - x_2^{c_2}, x_2^{c_2} - x_3^{c_3}, x_3^{c_3} - x_4^{c_4}\}.$ Since, in this case, $I_\mathcal{A} = C_\mathcal{A}$ by Theorem \ref{Th Result}, and $\mathcal{S}' = \varnothing,$ we are done.
\end{proof}

From the above proposition it follows that $I_\mathcal{A}$ is generic (see, e.g. \cite{Ojeda}) only in the CASE 1. The next theorem provides a minimal generating set for $I_\mathcal{A}$.

\begin{theorem}\label{Th Main}
The union of $\mathcal{S},$ the set $\mathcal{I}$ of all binomials $x_{i_1}^{u_{i_1}} x_{i_2}^{u_{i_2}} - x_{i_3}^{u_{i_3}} x_{i_4}^{u_{i_4}} \in I_\mathcal{A}$ with $0 < u_{i_j} < c_j,\ j = 1,2$, $u_{i_3}>0$, $u_{i_4}>0$ and $x_{i_3}^{u_{i_3}} x_{i_4}^{u_{i_4}}$ indispensable, and the set $\mathcal{R}$ of all binomials $x_1^{u_1} x_2^{u_2} - x_3^{u_3} x_4^{u_4} \in I_\mathcal{A} \setminus \mathcal{I}$ with full support such that
\begin{itemize}
\item $u_1 \leq c_1$ and $x_3^{u_3} x_4^{u_4}$ is indispensable, in the CASES 2(a) and 4(b).
\item $u_1 \leq c_1$ and/or $u_3 \leq c_3$ and there is no $x_1^{v_1} x_2^{v_2} - x_3^{v_3} x_4^{v_4} \in I_\mathcal{A}$ with full support such that $x_1^{v_1} x_2^{v_2}$ properly divides $x_1^{u_1 + \alpha c_1} x_2^{u_2 - \alpha c_2}$ or $x_3^{v_3} x_4^{v_4}$ properly divides $x_3^{u_3 + \alpha c_3} x_4^{u_4 - \alpha u_4}$ for some $\alpha \in \mathbb{N},$ in the CASE 2(b).
\end{itemize}
is a minimal system of generators of $I_\mathcal{A}$ (up to permutation of indices).
\end{theorem}

\begin{proof}
By Proposition \ref{Prop Critical3}, there exists a minimal system of binomial generators $\mathcal{S} \cup \mathcal{S}'$ of $I_\mathcal{A}$ such that $\mathcal{S}$ is a minimal system of generators of $C_\mathcal{A}$ and $\mathrm{supp}(f) = \{1,2,3,4\},$ for every $f \in \mathcal{S}'.$ Moreover, since all the binomials in the set $\mathcal{I}$ are indispensable by Corollary \ref{Cor indispensablebinom3}, we have $\mathcal{S}' = \mathcal{I} \cup \mathcal{R},$ where $\mathcal{R}$ is a set of binomials of $I_\mathcal{A}$ of the form $x_{i_1}^{u_{i_1}} x_{i_2}^{u_{i_2}} - x_{i_3}^{u_{i_3}} x_{i_4}^{u_{i_4}}$ with $u_{i_j} \neq 0$, for every $j$, and $u_{i_j} \geq c_j$ for some $j.$

Observe that if $\mathcal{R} = \varnothing,$ then the set defined in the statement of the theorem coincides with $\mathcal{S} \cup \mathcal{S}'$ and therefore it is a minimal set of generators. So, we assume that $\mathcal{R} \neq \varnothing,$ that is to say, there exists a minimal generator $x_1^{u_1} x_2^{u_2} - x_3^{u_3} x_4^{u_4} \in \mathcal{R}$ with $u_2 \geq c_2$ (by permuting variables if necessary). By Lemma \ref{Lemma Critical3} (ii) it holds that $c_1 a_1 = c_2 a_2$, so in CASE 1 we have  $\mathcal{R} = \varnothing$ and therefore we are done. Moreover, if $c_2 a_2 = c_i a_i,$ for an $i \in \{3,4\},$ then $x_1^{u_1} x_2^{u_2} \longleftrightarrow x_1^{u_1} x_2^{u_2-c_2} x_i^{c_i} \longleftrightarrow x_3^{u_3} x_4^{u_4}$ is a path in $G_b(I_\mathcal{A}),$ where $ b = u_1 a_1 + u_2 a_2,$ a contradiction with Theorem \ref{Th OjVi1}. Therefore, we conclude that the theorem is also true in CASE 2(c) and CASE 3. Notice that, in the CASE 4(a), we can proceed similarly to reach a contradiction; indeed, since $x_2^{c_2} - \mathbf{x}^\mathbf{v} \in \mathcal{S},$ where $\mathrm{supp}(\mathbf{x}^\mathbf{v})=\{3,4\}$, then $x_1^{c_1} - \mathbf{x}^\mathbf{v} \in I_\mathcal{A}$ and therefore $x_1^{u_1} x_2^{u_2} \longleftrightarrow x_1^{u_1+c_1} x_2^{u_2-c_2} \longleftrightarrow x_1^{u_1} x_2^{u_2-c_2} \mathbf{x}^\mathbf{v} \longleftrightarrow x_3^{u_3} x_4^{u_4}$ is a path in $G_b(I_\mathcal{A}),$ a contradiction with Theorem \ref{Th OjVi1}. Thus $\mathcal{R} = \varnothing$ in CASE 4(a), too.

Suppose now that $x_1^{v_1} x_i^{v_i} - x_2^{v_2} x_j^{v_j} \in \mathcal{R}.$ By Lemma \ref{Lemma Critical3} (ii) again, we obtain that at least one of the equalities $c_1 a_1 = c_i a_i$ and $c_2 a_2 = c_j a_j$ holds. But, as we proved above, these equalities are incompatible with the condition $x_1^{u_1} x_2^{u_2} - x_3^{u_3} x_4^{u_4} \in \mathcal{R}$ with $u_2 \geq c_2.$ Hence, all the binomials in $\mathcal{R}$ are of the form $x_1^\bullet x_2^\bullet - x_3^\bullet x_4^\bullet$ and $x_2$ arises, with exponent greater than or equal to $2$, in at least one of the variables.

We distinguish the following cases:

\framebox{CASE 2(a) or 4(b).} If there exists $x_1^{v_1} x_2^{v_2} - x_3^{v_3} x_4^{v_4} \in \mathcal{R}$ such that for instance $v_4 \geq c_4$, then $c_3 a_3 = c_4 a_4$ by Lemma \ref{Lemma Critical3} (ii). This is clearly incompatible with CASES 2(a) and 4(b), since $x_3^{v_3} x_4^{v_4} \longleftrightarrow x_3^{v_3} x_4^{v_4-c_4} \mathbf{x}^{\mathbf{u}_4} \longleftrightarrow x_1^{v_1} x_2^{v_2}$ is a path in $G_d(I_\mathcal{A}),\ d = a_1 v_1 + a_2 v_2,$ a contradiction with Theorem \ref{Th OjVi1}. Thus the binomials in $\mathcal{R}$ are of the form $x_1^{u_1} x_2^{u_2} - x_3^{u_3} x_4^{u_4}$ with $u_i < c_i,\ i = 3,4.$ If $x_3^{u_3} x_4^{u_4}$ is not indispensable, then there exists $x^\mathbf{v} - x_3^{v_3} x_4^{v_4} \in I_\mathcal{A}$ such that $0 < v_i \leq u_i,$ for $i = 3,4$, with at least one inequality strict and $\mathrm{supp}(x^\mathbf{v}) \subseteq \{1,2\}.$ So, $x_3^{u_3} x_4^{u_4} \longleftrightarrow x_3^{u_3-v_3} x_4^{u_4-v_4} \mathbf{x}^{\mathbf{v}} \longleftrightarrow x_1^{u_1} x_2^{u_2}$ is a path in $G_b(I_\mathcal{A})$ where $ b = a_3 u_3 + a_4 u_4,$ a contradiction with Theorem \ref{Th OjVi1}. Moreover, since $x_1^{c_1} - x_2^{c_2} \in I_\mathcal{A}$, we may change, if it is necessary, $\mathcal{R}$ by replacing every binomial $x_1^{u_1} x_2^{u_2} - x_3^{u_3} x_4^{u_4}$, where $u_1>c_1$, with $x_1^{u_1-\alpha c_1} x_2^{u_2 + \alpha c_2} - x_3^{u_3} x_4^{u_4} \in I_\mathcal{A}$ such that $0 < u_1-\alpha c_1 \leq c_1$ and $u_2 + \alpha c_2 \geq c_2.$ Now the new set $\mathcal{S} \cup \mathcal{I} \cup \mathcal{R}$ has the desired form. We have that $$x_1^{u_1} x_2^{u_2} - x_3^{u_3} x_4^{u_4}=(x_1^{u_1-\alpha c_1} x_2^{u_2 + \alpha c_2} - x_3^{u_3} x_4^{u_4})+x_1^{u_1-\alpha c_1} x_2^{u_2}(x_1^{\alpha c_1}-x_2^{\alpha c_2}),$$ so $\mathcal{S} \cup \mathcal{I} \cup \mathcal{R}$ is a generating set of $I_\mathcal{A}$. To see that this is actually minimal, by indispensability reasons, it suffices to show that if $x_1^{u_1} x_2^{u_2} - x_3^{u_3} x_4^{u_4} \in \mathcal{R}$ and $x_1^{v_1} x_2^{v_2} - x_3^{u_3} x_4^{u_4} \in \mathcal{S} \cup \mathcal{I} \cup \mathcal{R},$ then $x_1^{u_1} x_2^{u_2} = x_1^{v_1} x_2^{v_2}.$ Otherwise $x_1^{u_1} x_2^{u_2} - x_1^{v_1} x_2^{v_2} \in I_\mathcal{A},$ but $0 < u_1 \leq c_1$ and $v_1 \leq c_1.$ Thus $\vert u_1 - v_1 \vert \leq c_1,$ so $u_1 = c_1, v_1 = 0$ and therefore $v_2 = c_2$, since every binomial in $\mathcal{S} \cup \mathcal{I} \cup \mathcal{R}$ with cardinality less than four is critical. We have that $c_1 a_1 + a_2 u_2 = c_2 a_2$ and also $c_1 a_1 = c_2 a_2$, so $u_2 = 0$ a contradiction.

\framebox{CASE 2(b).} Now, by modifying $\mathcal{R}$ as in the previous case if necessary, we have that the binomials in $\mathcal{R}$ are of the following form: $x_1^{u_1} x_2^{u_2} - x_3^{u_3} x_4^{u_4}$ with $0 < u_1 \leq c_1, u_2 \neq 0$ and/or $0 < u_3 \leq c_3, u_4 \neq 0.$ If there exists $\alpha \in \mathbb{N}$
and $x_1^{v_1} x_2^{v_2} - x_3^{v_3} x_4^{v_4} \in I_\mathcal{A}$ with full support such that  $x_1^{u_1+\alpha c_1} x_2^{u_2-\alpha c_2} = m x_1^{v_1} x_2^{v_2}$ (or $x_3^{u_3+\alpha c_3} x_4^{u_4-\alpha c_4} = m x_3^{v_3} x_4^{v_4},$ respectively)  with $m \neq 1,$ then $x_1^{u_1} x_2^{u_2} \longleftrightarrow m x_3^{v_3} x_4^{v_4} \longleftrightarrow x_3^{u_3} x_4^{u_4}$ (or $x_1^{u_1} x_2^{u_2} \longleftrightarrow x_1^{v_1} x_2^{v_2} m \longleftrightarrow x_3^{u_3} x_4^{u_4},$ respectively) is a path in $G_b(I_\mathcal{A})$, where $b = u_1 a_1 + u_2 a_2,$ a contradiction with Theorem \ref{Th OjVi1}. So, we conclude that all the binomials in $\mathcal{R}$ are of the desired form. Moreover, given $f = x_1^{u_1} x_2^{u_2} - x_3^{u_3} x_4^{u_4} \in \mathcal{R}$ and a monomial $\mathbf{x}^\mathbf{v}$ with $\deg_\mathcal{A}(\mathbf{x}^\mathbf{v})= u_1 a_1 + u_2 a_2 ,$ then either $v_1 = v_2 = 0$ or $v_1=v_3 = v_4 = 0$ and $v_2>c_2.$ Indeed, since  $x_1^{u_1} x_2^{u_2} - x_1^{v_1} x_2^{v_2} x_3^{v_3} x_4^{v_4} \in I_\mathcal{A},$ then
\begin{itemize}
\item[(i)] $g = x_1^{u_1-v_1} x_2^{u_2-v_2} - x_3^{v_3} x_4^{v_4} \in I_\mathcal{A},$ when $v_1 \leq u_1$ and $v_2 <u_2.$ If $g$ has full support, then $v_1=v_2 =0,$ otherwise $f \not\in \mathcal{R}.$ If for instance $u_1-v_1 = 0,$ then $u_2-v_2 \geq c_2,$ because of the minimality of $c_2.$ Thus, $g' = x_1^{u_1-v_1+c_1} x_2^{u_2-v_2-c_2} - x_3^{v_3} x_4^{v_4} \in I_\mathcal{A}.$ If $g'$ has full support, then $v_1=v_2 =0;$ otherwise the monomial $x_1^{u_1-v_1+c_1} x_2^{u_2-v_2-c_2}$ properly divides $x_1^{u_1+c_1} x_2^{u_2-c_2},$ that is to say,  $f \not\in \mathcal{R}.$ If $g'$ does not have full support, say $v_3 = 0,$ then $v_4 \geq c_4$ (due to the minimality of $c_4$). So, we may define $g'' =  x_1^{u_1-v_1+c_1} x_2^{u_2-v_2-c_2} - x_3^{c_3} x_4^{v_4-c_4} \in I_\mathcal{A}$ and conclude that $v_1 = v_2 = 0,$ as before.
\item[(ii)] $g = x_1^{u_1-v_1} - x_2^{v_2-u_2} x_3^{v_3} x_4^{v_4} \in I_\mathcal{A},$ when $v_1<u_1$ and $v_2 \geq u_2.$ Since $0 < u_1 \leq c_1,$ we have that $v_1 =0$ and also $u_1=c_1$. Thus $v_2-u_2 = c_2$ and $v_3 = v_4 = 0$, since $x_1^{c_1} - x_2^{c_2}$ is the only critical binomial with respect to $x_1$.
\item[(iii)] $g = x_2^{u_2-v_2} - x_1^{v_1-u_1} x_3^{v_3} x_4^{v_4} \in I_\mathcal{A},$ when $v_1 \geq u_1$ and $v_2<u_2.$ Now, by the minimality of $c_2,$ we have that $u_2 - v_2 \geq c_2$ and therefore $h = x_1^{c_1} x_2^{u_2-v_2-c_2} - x_1^{v_1-u_1} x_3^{v_3} x_4^{v_4} \in I_\mathcal{A}.$ So, either $x_1^{c_1+u_1-v_1} x_2^{u_2-v_2-c_2} - x_3^{v_3} x_4^{v_4} \in I_\mathcal{A},$ when $c_1 \geq v_1-u_1,$ or $x_2^{u_2-v_2-c_2} - x_1^{v_1-u_1-c_1} x_3^{v_3} x_4^{v_4} \in I_\mathcal{A},$ when $c_1 < v_1-u_1.$ In the first case we proceed as in (i), while in the other we repeat the same argument and so on. This process can not continue indefinitely, since there exists $\alpha \in \mathbb{N}$ such that $\alpha c_1 < v_1-u_1$, and thus we are done.
\end{itemize}
From Theorem \ref{Th OjVi1} we have that there exists a minimal generator of $\mathcal{A}-$degree $\deg_\mathcal{A}(f)$ for each $f \in \mathcal{R}.$ Furthermore, by direct checking one can show that all the binomials in $\mathcal{I} \cup \mathcal{R}$ have a different $\mathcal{A}-$degree, and all these $\mathcal{A}-$degrees are different from both $c_1 a_1 $ and $c_2 a_2.$ Thus, we conclude that $\mathcal{S} \cup \mathcal{I} \cup \mathcal{R}$ is a minimal system of generators of $I_\mathcal{A}.$
\end{proof}

Combining Theorem \ref{Th Main} with Corollaries \ref{Cor Critical} and \ref{Cor indispensablebinom3} yields the following theorem.

\begin{theorem}\label{Thm Critical1} With the same notation as in Theorem \ref{Th Main}, the ideal $I_\mathcal{A}$ has a unique minimal system of generators if, and only if, $C_\mathcal{A}$ has a unique minimal system of generators and $\mathcal{R} = \varnothing.$
\end{theorem}

In \cite{Ojeda}, it is shown that there exist semigroup ideals of $\mathbbmss{k}[x_1, \ldots, x_4]$ with unique minimal system of binomial generators of cardinality $m,$ for every $m \geq 7.$

\begin{example} Let $\mathcal{A}=\{6, 8, 17, 19\}$. The critical binomial $x_1^{4}-x_2^{3}$ of $I_\mathcal{A}$ is indispensable, while the critical binomial $x_4^{2}-x_1x_2^{4}$ is not indispensable. Thus we are in CASE 4(b). The binomial
$x_1^{2}x_2^{3}-x_3x_4$ belongs to $\mathcal{R}$ and therefore, from Theorem \ref{Thm Critical1}, the toric
ideal $I_{\mathcal{A}}$ does not have a unique minimal system of binomial generators.
\end{example}

\begin{example} Let $\mathcal{A}=\{25, 30, 57, 76\}$, then the minimal number of generators of $I_\mathcal{A}$ equals $8$. The only critical binomials of $I_\mathcal{A}$ are
$\pm(x_1^{6}-x_2^{5})$ and $\pm(x_3^{4}-x_4^{3})$, so we are in CASE 2(b). The binomial
$x_1^{3}x_2^{7}-x_3x_4^{3}$ belongs to $\mathcal{R}$ and therefore, from Theorem \ref{Thm Critical1}, the toric
ideal $I_{\mathcal{A}}$ does not have a unique minimal system of binomial generators.
\end{example}

Observe that $I_\mathcal{A}$ is a complete intersection only in cases 2(a-c), 3 and 4(b). Moreover, except from 2(b), in all the other cases $I_\mathcal{A} = C_\mathcal{A}.$ In the case 2(b) a minimal system of binomial generators is $x_1^{c_1} - x_2^{c_2}, x_3^{c_3} - x_4^{c_4}$ and $x_1^{u_1} x_2^{u_2} - x_3^{u_3} x_4^{u_4}$ where $a_1 u_1 + a_2 u_2 = a_3 u_3 + a_4 u_4 = \mathrm{lcm}(\mathrm{gcd}(a_1,a_2), \mathrm{gcd}(a_3,a_4) )$ (see, \cite{Delorme}).

It is well known that the ring $\mathbbmss{k}[\mathbf{x}]/I_\mathcal{A}$ is Gorenstein if and only if the semigroup $\mathbb{N}\mathcal{A}$ is symmetric, see \cite{Ku}. We will prove that if $\mathbb{N}\mathcal{A}$ is symmetric and $I_\mathcal{A}$ is not a complete intersection, then $I_\mathcal{A}$ has a unique minimal system of binomial generators.

\begin{theorem} \label{Cr1}
If $f_1=x_{1}^{c_1}-x_{3}^{u_{13}}x_{4}^{u_{14}}, \ f_2=x_{2}^{c_2}-x_{1}^{u_{21}}x_{4}^{u_{24}},\ f_3=x_{3}^{c_3}-x_{1}^{u_{31}}x_{2}^{u_{32}}$ and $f_4=x_{4}^{c_4}-x_{2}^{u_{42}}x_{3}^{u_{43}}$ are critical binomials of $I_\mathcal{A}$ such that $\mathrm{supp}(f_i)$ has cardinality equal to $3$, for every $i \in \{1,\ldots,4\}$, then $I_\mathcal{A}$ has a unique minimal system of binomial generators.
\end{theorem}

\begin{proof} We have that every exponent $u_{ij}$ of $x_j$ is strictly less than $c_j$, for each $j=1,\ldots,4$. If for instance $u_{13} \geq c_3,$ then $x_1^{c_1}-x_1^{u_{31}}x_2^{u_{32}}x_{3}^{u_{13}-c_3}x_4^{u_{14}}=f_1+x_{3}^{u_{13}-c_3}x_{4}^{u_{14}}f_3 \in I_\mathcal{A}$ and therefore $x_1^{c_1-u_{31}}-x_{2}^{u_{32}}x_3^{u_{13}-c_3}x_4^{u_{14}} \in I_\mathcal{A},$ a contradiction to the minimality of $c_1.$ By Proposition \ref{Prop critical0} we have that $c_ia_i \neq c_ja_j,$ for every $i \neq j.$ We will prove that every $f_i$ is indispensable of $C_\mathcal{A}.$ Suppose for example that $f_1$ is not indispensable of $C_\mathcal{A},$ then there is a binomial $g=x_1^{c_1}-x_2^{v_2}x_3^{v_3}x_4^{v_4} \in I_\mathcal{A}.$ So $x_{3}^{u_{13}}x_{4}^{u_{14}}-x_2^{v_2}x_3^{v_3}x_4^{v_4} \in I_\mathcal{A}$, and thus $v_3<u_{13}$ and $v_4<u_{14}$, since
$u_{13}<c_3$ and $u_{14}<c_4.$ We have that $x_2^{v_2}-x_{3}^{u_{13}-v_3}x_{4}^{u_{14}-v_4} \in I_\mathcal{A}$ and also $x_1^{c_1}-x_1^{u_{21}}x_2^{v_2-c_2}x_3^{v_3}x_4^{u_{24}+v_4}=g+x_2^{v_2-c_2}x_3^{v_3}x_4^{v_4}f_2 \in I_\mathcal{A}.$ Therefore $x_1^{c_1-u_{21}}-x_2^{v_2-c_2}x_3^{v_3}x_4^{u_{24}+v_4} \in I_\mathcal{A},$ a contradiction to the minimality of $c_1.$ Analogously we can prove that $f_2,$ $f_3$ and $f_4$ are indispensable of $C_\mathcal{A}.$ Thus $C_\mathcal{A}$ is generated by its indispensable binomials and therefore, from Theorem \ref{Thm Critical1}, the toric ideal $I_\mathcal{A}$ has a unique minimal system of binomial generators.
\end{proof}

\begin{corollary}\label{Cor3.13}
Let $\mathbb{N}\mathcal{A}$ be a symmetric semigroup. If $I_\mathcal{A}$ is not a complete intersection, then it has a unique minimal system of binomial generators.
\end{corollary}

\begin{proof}
From Theorem 3 in \cite{Bresinsky75} the toric ideal $I_\mathcal{A}$ has a minimal generating set consisting of five binomials, namely four critical binomials of the form defined in the above theorem and a non critical binomial. By Theorem \ref{Cr1} the toric ideal $I_\mathcal{A}$ is generated by its indispensable binomials.
\end{proof}

According to \cite[Theorem]{Bresinsky75} the integers $a_i$ are polynomials in the exponents of the binomial in a minimal generating system of $I_\mathcal{A}$. We can see these expressions as a system of four polynomial equations, which in light of Corollary \ref{Cor3.13}, has a unique solution over the positive integers.

\begin{remark}
In \cite[Theorem 6.4]{komeda}, it is shown that if $\mathbb N \mathcal A$ is pseudo-symmetric (see \cite{GSRo} for a definition), then $f_1 = x_1^{c_1} - x_3 x_4^{c_4-1}, f_2 = x_2^{c_2} - x_1^{u_{21}} x_4, f_3 = x_3^{c_3} - x_1^{c_1-u_{21}-1} x_2, f_4 = x_4^{c_4} - x_1 x_2^{c_2-1} x_3^{c_3-1}$ and $g = x_1^{u_{21}+1} x_3^{c_3-1} - x_2 x_4^{c_4-1}$ with $c_i > 1, i = 1, \ldots, 4$, and $u_{21}-1 < c_1$, is a minimal system of generators of $I_\mathcal{A}$. Now, an easy check shows that $c_i a_i \neq c_j a_j$ for every $i \neq j$. The interested reader may prove that $C_\mathcal{A}$ has a unique minimal system of generators if and only if $u_{21} = c_1 - 2$. Thus, since $\mathcal R = \varnothing$, by Theorem \ref{Thm Critical1}, we conclude that $I_\mathcal{A}$ is generated by its indispensable binomials if and only if $c_2 n_2 \neq (c_1-2) n_1 + n_4$.
\end{remark}

If the cardinality of $\mathcal{A}$ is greater than $4$, the analogous of  Corollary \ref{Cor3.13} is not true in general. In \cite{Rosales} it is shown that the semigroup generated by $\mathcal{A} = \{15, 16, 81, 82, 83, 84\}$ is symmetric. Since the monomials $x_1^{11}, x_3 x_6$ and $x_4 x_5$ have the same $\mathcal{A}-$degree, we conclude, by Theorem \ref{Th OjVi1}, that the ideal $I_\mathcal{A}$ does not have a unique minimal system of binomial generators.

\section*{Acknowledgments}

Part of this work was done during a visit of the first author to the University of Extremadura financed by the Plan Propio 2010 of the University of Extremadura. We thank the referee for helpful comments and suggestions that improved the paper.


\end{document}